\documentclass[a4paper,12pt]{article}
\usepackage[utf8]{inputenc}

\usepackage{fullpage}

\usepackage{macros}

\usepackage{tikz}

\usepackage{algorithm}
\usepackage{algorithmic}

\usepackage{multirow}
\usepackage{booktabs} 

\newtheorem{assumption}{Assumption}
\newtheorem{definition}{Definition}

\newcommand{\assref}[1]{Assumption~\ref{ass:#1}}

\DeclareMathOperator{\conv}{Conv}

\def\ug{\bu_\gamma}

\def\etad{\eta_{\ast}}

\def\uhat{\hat{\bu}}
\def\u0hat{\hat{u_0}}

\def\fg{\bff_\gamma}

\def\alphahat{\hat{\bsalpha}}
\def\alphahatu{\alphahat(\bu)}
\def\alphahatg{\alphahat_\gamma}
\def\alphahatgu{\alphahatg(\bu)}
\def\hg{h_\gamma}
\def\jg{j_\gamma}
\def\rg{\br_\gamma}
\def\Ga{G_{\bsalpha}}
\def\Gb{G_{\bsbeta}}
\def\Gu{G_{\alphahatu}}
\def\Ggu{G_{\alphahatgu}}

\begin{document}



\title{On the convergence of the regularized entropy-based moment method for 
kinetic equations}

\author{Graham W.\ Alldredge\thanks{The first author was supported in part by by 
the German Science  Foundation (DFG), project ID AL 2030/1-1.} \\
Berlin, Germany \\
\texttt{gwak@posteo.net}
\and
Martin Frank \\
Department of Mathematics \\
Karlsruhe Institute of Technology \\
Karlsruhe, Germany \\
\texttt{martin.frank@kit.edu}
\and
Jan Giesselmann\thanks{The third author is grateful for financial support  
by the German Science Foundation (DFG) via grant TRR 154 (Mathematical 
modelling, simulation and optimization using the example of gas networks) 
project C05.} \\
Department of Mathematics,
Technical University of Darmstadt \\
Darmstadt, Germany \\
\texttt{giesselmann@mathematik.tu-darmstadt.de}
}

\maketitle

\begin{abstract}
The entropy-based moment method is a well-known discretization for the velocity 
variable in kinetic equations which has many desirable theoretical properties 
but is difficult to implement with high-order numerical methods. 
The regularized entropy-based moment method was recently introduced to remove 
one of the main challenges in the implementation of the entropy-based moment 
method, namely the requirement of the realizability of the numerical solution.
In this work we use the method of relative entropy to prove the convergence of 
the regularized method to the original method as the regularization parameter 
goes to zero and give convergence rates.
Our main assumptions are the boundedness of the velocity domain and that the 
original moment solution is Lipschitz continuous in space and bounded away from 
the boundary of realizability.
We provide results from numerical simulations showing that the convergence 
rates we prove are optimal.
\end{abstract}

\section{Introduction}

Kinetic equations model systems consisting of  large numbers of particles that
interact with each other or with a background medium and arise in a wide 
variety of applications including
rarefied gas dynamics \cite{Cercignani},
neutron transport \cite{Lewis-Miller-1984},
radiative transport \cite{mihalas1999foundations},
and semiconductors \cite{markowich1990}.
The numerical solution of kinetic equations remains an area of active research.
In this work, we consider the entropy-based moment method \cite{Lev96}, which 
is a discretization of the velocity variable in the kinetic equation.
It has many desirable theoretical properties but is computationally expensive 
and challenging to implement.

Recently in \cite{AllFraHau19} a regularized version of the entropy-based 
moment equations was proposed to simplify the implementation of numerical 
methods for the entropy-based moment equations.
These regularized entropy-based moment equations require the selection of a 
regularization parameter, and the authors in \cite{AllFraHau19} proposed a 
rule for selecting the regularization parameter so that the error introduced by 
the regularization was of the order of the error in the spatiotemporal 
discretization.
With this selection rule the authors produced numerical results which 
showed that the regularized equations could be used to compute accurate results 
of the original entropy-based moment equations.

In this work, we prove that exact solutions of the regularized entropy-based 
moment equations converge to the solutions of the original equations.
We quantify the difference between these two solutions using the relative entropy, 
and the convergence rate is quadratic in the regularization parameter.
Under some assumptions, this is equivalent to linear convergence in the 
regularization parameter in the $L^2$ norm.

To obtain our results we require the following main assumptions:
\begin{enumerate}[(i)]
 \item the set of velocities of the kinetic equation is bounded;
 \item periodic boundary conditions;
 \item Lipschitz continuity of the function defining the collision term in the 
  original entropy-based moment equations;
 \item the solution of the original entropy-based moment equations is both
  Lipschitz continuous in space and bounded away from the boundary of the set 
  of realizable moment vectors.
\end{enumerate}

The relative-entropy techniques we use are very similar to what was done in 
\cite{BerthelinVasseur2005,BerthelinTzavarasVasseur2009,
GiesselmannTzavaras2017,Tzavaras2005}.
In general, relative entropy estimates are a widely applicable tool for 
comparing thermomechanical theories having the form of hyperbolic balance laws 
that are endowed with a strictly convex entropy 
\cite{GiesselmannLattanzioTzavaras2017}.
A general limitation of this methodology is that it requires the solution to the 
limiting system to be Lipschitz continuous - this property can (usually) only be 
expected for short times since shocks may form.
There is recent progress in overcoming this limitation, at least in one space 
dimension, by using the relative entropy with shifts methodology that was 
developed by  Vasseur and co-workers \cite{Krupa2019,SerreVasseur2016}. 
However, this condition can most probably not be removed in  two or more space 
dimensions since it is connected to non-uniqueness of entropy solutions 
to hyperbolic balance laws.
% ...

To present our results, we first introduce the entropy-based moment equations 
and their regularized version in \secref{equations} and precisely state our 
assumptions.
Then in \secref{main-results} introduce the technique of relative entropy and 
give a general version of our main result.
We prove the estimates upon which our main result relies and give the subsets 
of the realizable set on which they hold in \secref{estimates}.
Next we present the results of numerical experiments confirming the 
theoretically predicted rates of convergence in \secref{num-results}, and 
finally we draw conclusions and discuss directions for future work in 
\secref{conc}.

\section{Entropy-based moment equations and regularization}
\label{sec:equations}

\subsection{The kinetic equation}

Kinetic equations evolve the \textit{kinetic density function}
${f \colon [0, \infty) \times X \times V \to [0, \infty)}$ according to
\begin{equation}\label{eq:kinetic}
 \partial_t f(t, x, v) + v \cdot \nabla_x f(t, x, v) = \cC(f(t, x, \cdot))(v)
\end{equation}
(when neglecting long-range interactions).
The function $f$ depends on time $t \in [0, \infty)$, position
${x \in X \subseteq \R^d}$, and a velocity variable $v \in V \subseteq \R^d$.
The operator $\cC$ introduces the effects of particle collisions; at each $x$ 
and $t$, it 
is an integral operator in $v$.
In order to be well-posed, \eqref{eq:kinetic} must be accompanied by appropriate
initial and boundary conditions.

The results in this work depend strongly on the following assumption.
\begin{assumption}\label{ass:v-bnd}
The set of velocities $V$ is bounded.
\end{assumption}

We will see the crucial consequences of this assumption in the next subsection.
For any $g \in L^1(V)$ we use the notation
\begin{align}
 \Vint{g} := \int_V g(v) \intdv.
\end{align}
We define
\begin{align}
 |V| := \Vint{1} = \int_V \intdv\,.
\end{align}
Many of the constants below depend on $\max_{v \in V} \|v\|$, though for 
clarity of exposition we do not give this dependence explicitly.

We consider kinetic equations where the collision operator satisfies an 
entropy-dissipation property:
Let $\eta : D \to \R$, where $D \subseteq \R$, be a strictly convex function 
and $\bbF(V) := \{ g \in L^1(V) : \range(g) \subseteq D \}$.
We call $\eta$ the \emph{kinetic entropy function}.
Then the local entropy $\cH : \bbF(V) \to \R$ is given by
\begin{align}\label{eq:cH}
 \cH(f) := \Vint{\eta(f)}.
\end{align}
This entropy is dissipated if the collision operator $\cC$ satisfies
\begin{align}\label{eq:entropy-diss-kin}
 \Vint{\eta'(g) \cC(g)} \le 0
\end{align}
for all $g : V \to D$ such that the integral is defined.
Furthermore, we assume that $\eta$ is sufficiently smooth:

\begin{assumption}\label{ass:smooth-eta}
The kinetic entropy satisfies $\eta \in C^3(D)$, $\eta'' > 0$ on $D$, and 
\eqref{eq:entropy-diss-kin}, i.e., $\eta$ is an entropy dissipated by the 
kinetic equation \eqref{eq:kinetic}.
\end{assumption}

\subsection{The original entropy-based moment equations and realizability}

The original entropy-based moment equations are a semidiscretization of the 
kinetic equation \eqref{eq:kinetic} in the velocity variable $v$.
For an overview, see \cite{Lev96}.
The velocity dependence of $f$ at each point in time and space is replaced by 
the vector of moments
\begin{align}
 \bu(t, x) := (u_0(t, x), u_1(t, x), \ldots , u_N(t, x)) \in \R^{N + 1},
\end{align}
which contains the approximations of velocity integrals of $f$ multiplied by 
the basis functions
\begin{align}
 \bm(v) := (m_0(v), m_1(v), \ldots , m_N(v)),
\end{align}
that is, ${u_i(t, x) \simeq \vint{m_i f(t, x, \cdot)}}$ for all
${i \in \{ 0, 1, \ldots , N \}}$.
Usually the basis functions are polynomials.
We make the following assumptions on the basis functions:

\begin{assumption}\label{ass:basis-functions}
\begin{enumerate}[(i)]
 \item For every $i \in \{0, 1, \ldots , N\}$ we have $m_i \in L^\infty(V)$.
  Without loss of generality, we assume that $\|m_i\|_{L^\infty(V)}$ is bounded 
  by one:
  \begin{align}\label{eq:m-bnd}
   \|m_i\|_{L^\infty(V)} = \sup_{v \in V} |m_i(v)| \le 1
    \qquad
   \text{for } i \in \{0, 1, \ldots , N \}.
  \end{align}
 \item The constant function is in the linear span of the basis functions.
  Without loss of generality we assume $m_0(v) \equiv 1$.
\end{enumerate}
\end{assumption}

For each moment vector, the entropy-based moment method reconstructs an ansatz 
for the the kinetic density by solving the constrained optimization problem
\begin{equation}\label{eq:primal}
 \minimize_{g \in \bbF(V)} \: \cH(g)
  \qquad \st \: \Vint{\bm g} = \bu
\end{equation}
(recall the definition of $\cH$ in \eqref{eq:cH}).
Under \assref{v-bnd}, this problem has a unique solution for every $\bu \in \cR$
\cite{Jun00}, where
\begin{align}
 \cR := \left\{ \bu : \text{ there exists a } g \in \bbF(V)
  \text{ such that } \Vint{\bm g} = \bu \right\}
\end{align}
is the set of all \emph{realizable} moment vectors, and consequently, the 
system of moment equations is well-defined and hyperbolic for all realizable 
moment vectors.%
\footnote{
When $V$ is not bounded, such as in entropy-based moment equations for the 
Boltzmann equation for rarefied gas dynamics, where $V = \R^3$, there are 
important examples of realizable moment vectors for which the primal problem 
has no solution; see \cite{Junk-1998,Jun00,Hauck-Levermore-Tits-2008}.
This is a significant open problem for entropy-based moment equations, and 
neither the regularization nor our work here can get around this issue.
}

The solution to \eqref{eq:primal} takes the form $\Gu$, where
\begin{align}\label{eq:ansatz}
\Ga := \eta'_*(\bsalpha \cdot \bm)
\end{align}
and $\alphahat: \cR \to \R^{N + 1}$ maps a moment vector $\bu$ to the solution 
of the dual problem 
\begin{equation}\label{eq:dual}
\alphahat(\bu) = \argmax_{\bsalpha \in \R^{N + 1}} 
 \left\{ \bsalpha \cdot \bu - \Vint{\etad(\bsalpha \cdot \bm)}\right\};
\end{equation} 
here $\etad$ is the Legendre dual%
\footnote{
See, e.g., \cite[\S 3.3.2.]{evans2010partial} or \cite[\S 3.3]{boyd2004convex}, 
where what we call the Legendre dual is called the conjugate function.
}
of $\eta$ (and thus $\etad'$ is the inverse function of $\eta'$).
The components of $\bsalpha$ are the Lagrange multipliers for the primal 
problem \eqref{eq:primal}.

Now we are ready to give the entropy-based moment equations:
\begin{subequations}\label{eq:mn}
\begin{align}
 \partial_t \bu + \nabla_x \cdot \bff(\bu) &= \br(\bu),
\end{align}
where the flux function $\bff$ and relaxation term $\br$ are given by
\begin{align}\label{eq:f-and-r}
 \bff(\bu) := \Vint{v \bm \Gu} \qquand
 \br(\bu) := \Vint{\bm \cC(\Gu)}.
\end{align}
\end{subequations}
Classical solutions of the moment equations \eqref{eq:mn} satisfy the entropy 
dissipation law
\begin{align}\label{eq:entropy-diss-mn}
 \partial_t h(\bu) + \nabla_x \cdot j(\bu) = h'(\bu) \cdot \br(\bu) \le 0
\end{align}
for the entropy and entropy flux
\begin{equation}\label{eq:entropy-entropy-flux}
 h(\bu) := \Vint{\eta(\Gu)}  \quand
 j(\bu) := \Vint{v \eta(\Gu)}.
\end{equation}
Note that $h' = \alphahat$, and that $h$ is strictly convex as a consequence of 
\assref{smooth-eta} \cite{Hauck-Levermore-Tits-2008}.
For readers unfamiliar with the derivations of the dual problem and the entropy 
dissipation law and related properties of the entropy-based moment equations, 
we review these in Appendix \ref{sec:mn-review}.

We make the following assumption on the collision term:

\begin{assumption}\label{ass:r-lip}
The function $\br$ in the collision term of the original entropy-based moment 
equations \eqref{eq:mn} is Lipschitz continuous with constant $C_{\br}$ and 
satisfies $\lim_{\bu \rightarrow 0} \br(\bu) = 0$.
\end{assumption}

This applies, for example, to linear collision operators like in the case of 
isotropic scattering.
This assumption cannot be expected to apply for the Boltzmann collision 
operator, at least not globally, since it is quadratic.
The assumption $\lim_{\bu \rightarrow 0} \br(\bu) = 0$ is natural since 
otherwise there would be collision effects in the moment equations when no 
particles are present.

Notice that the flux $\bff$ and the source $\br$ in \eqref{eq:mn} can only be
defined when the optimization problem \eqref{eq:primal} is feasible, i.e., 
when the $\bu \in \cR$.
When $D = [0, \infty)$, this set corresponds to the set of vectors which 
contain the moments of a nonnegative density, which are indeed the only moment 
vectors we want to consider since the kinetic density $f$ should be nonnegative.
This theoretical advantage is however impractical for implementations because 
it is in general difficult to keep a high-order numerical solution of 
\eqref{eq:mn} within the realizable set $\cR$.

Finally, we note that a consequence of \assref{v-bnd} is that $\alphahat$ is a 
smooth bijection from $\cR$ to $\R^{N + 1}$.
Its inverse is the function which gives the moment vector of the entropy ansatz 
corresponding to a given multiplier vector:
\begin{align}
 \uhat(\bsalpha) := \Vint{\bm \etad'(\bsalpha \cdot \bm)}.
\end{align}
This function plays a role in the analysis later.

\subsection{The regularized entropy-based moment equations}

To work around the problem of realizability, the regularized entropy-based 
moment equations were proposed in \cite{AllFraHau19}.
Let $\gamma \in (0, \infty)$ be the regularization parameter.
Then the regularized entropy-based moment equations are given by
\begin{subequations}
\label{eq:reg-mn}
\begin{align}
\partial_t \bu + \nabla_x \cdot \bff_\gamma(\bu) &= \br_\gamma(\bu),
\end{align}
where
\begin{align}\label{eq:f-and-r-gam}
\bff_\gamma(\bu) := \Vint{v \bm \Ggu} \qquand
\br_\gamma(\bu) := \Vint{\bm \cC(\Ggu)}.
\end{align}
\end{subequations}
The ansatz $\Ggu$ has the same form as above (i.e.,
${\Ggu = \etad'(\alphahatgu \cdot \bm)}$) but is the solution of the 
unconstrained optimization problem
\begin{align}\label{eq:primal-reg}
\minimize_{g \in \bbF(V)} \: \Vint{\eta(g)} + \frac1{2\gamma}
 \left\| \Vint{\bm g} - \bu \right\|^2,
\end{align}
which is feasible for any moment vector $\bu \in \R^{N + 1}$ (again under
\assref{v-bnd}).
The new multiplier vector $\alphahatgu$ is the solution of the corresponding 
dual problem
\begin{align}\label{eq:reg-mult}
\alphahatgu := \argmax_{\bsalpha \in \R^{N + 1}} \left\{
\bsalpha \cdot \bu
- \Vint{\eta_*(\bsalpha \cdot \bm)}
- \frac\gamma2 \|\bsalpha\|^2 \right\}
\end{align}
and satisfies the first-order necessary conditions
\begin{align}\label{eq:1st-ord-necc}
 \bu = \uhat(\alphahatgu) + \gamma \alphahatgu.
\end{align}
Classical solutions of the regularized equations \eqref{eq:reg-mn} satisfy
\begin{align*}
 \partial_t \hg(\bu) + \nabla_x \cdot \jg(\bu)
  = \hg'(\bu) \cdot \br_\gamma(\bu) \le 0.
\end{align*}
where
\begin{equation}\label{eq:hg-jg}
 \hg(\bu) := \Vint{\eta(\Ggu)}
   + \frac1{2\gamma} \left\| \Vint{\bm \Ggu} - \bu \right\|^2
  \qquand
 \jg(\bu) := \Vint{v \eta(\Ggu)}.
\end{equation}
Analogously to the original case, we have $\hg' = \alphahatg$.
In this work, we are mostly concerned with entropy solutions of the regularized equations, i.e. weak solutions of \eqref{eq:reg-mn} satisfying the admissibility criterion
\begin{align}\label{eq:entropy-diss-rmn}
 \partial_t \hg(\bu) + \nabla_x \cdot \jg(\bu)
  \leq \hg'(\bu) \cdot \br_\gamma(\bu) .
\end{align}
Note that any Lipschitz continuous solution of \eqref{eq:reg-mn} is automatically an entropy solution.

Like $h$, the entropy $\hg$ of the regularized equations is convex, and its 
Legendre dual $(\hg)_*$ will prove to be useful in the analysis.
The Legendre dual $(\hg)_*$ and its first and second derivatives are given by 
\cite{AllFraHau19}
\begin{subequations}
\begin{align}
 (\hg)_*(\bsalpha) &= \Vint{\etad(\bsalpha \cdot \bm)}
  + \frac \gamma 2 \|\bsalpha\|^2, \\
 (\hg)'_*(\bsalpha) &= \Vint{\bm \etad'(\bsalpha \cdot \bm)}
  + \gamma \bsalpha, \text{ and} \\
 (\hg)''_*(\bsalpha) &= \Vint{\bm \bm^T \etad''(\bsalpha \cdot \bm)}
  + \gamma I, \label{eq:hg-hess}
\end{align}
\end{subequations}
where $I$ is the $(N + 1) \times (N + 1)$ identity matrix.
Note that $(\hg)'_* \circ \hg' = (\hg)'_* \circ \alphahatg = \id$, so we also 
have $\hg'' = ((\hg)''_* \circ \alphahatg)^{-1}$, where the inverse indicates 
the matrix inverse.
One also immediately recognizes from the form of $(\hg)''_*$ that $(\hg)_*$ and 
thus $\hg$ are strictly convex for any $\gamma > 0$.

\section{Relative entropy for convergence}
\label{sec:main-results}

Let $X \subset \R^d$ be a $d$-cube and $T \in (0, \infty)$.
We consider the initial-value problem
\begin{subequations}\label{eq:moment-eq}
\begin{align}
 \partial_t \bu + \nabla_x \cdot \bff(\bu) &= \br(\bu)
  & (t, x) &\in (0, T] \times X, \\
 \bu(0, x) &= \bu^0(x) & x &\in X,
\end{align}
\end{subequations}
where $\bu^0$ are the given initial conditions, and we use periodic boundary 
conditions in space.
For any $\gamma \in (0, \infty)$, the regularized moment equations are
\begin{subequations}\label{eq:reg-moment-eq}
\begin{align}
 \partial_t \ug + \nabla_x \cdot \fg(\ug) &= \rg(\ug)
  & (t, x) &\in (0, T] \times X, \\
 \ug(0, x) &= \bu^0(x) & x &\in X,
\end{align}
\end{subequations}
and we also take periodic boundary conditions.
Notice that the initial conditions are the same.

Following Dafermos \cite{Dafermos2016}, we introduce the relative entropy and 
relative entropy flux relative to $h_\gamma$:
\begin{definition}
Given moments $\ug, \, \bu \in \R^{N+1}$ the \emph{relative entropy} and 
\emph{relative entropy flux} are given by
\begin{align}
 \hg(\ug | \bu) &:= \hg(\ug) - \hg(\bu) - \alphahatgu \cdot (\ug - \bu)
  \text{ and} \label{eq:rel-hg} \\
 \jg(\ug | \bu) &:= \jg(\ug) - \jg(\bu)
  - \alphahatgu \cdot (\fg(\ug) - \fg(\bu)),
\end{align}
respectively.
\end{definition}
Note that using the relative entropy with respect to $h_\gamma$ allows us (in 
principle) to use non-realizable moment vectors in both arguments of the 
relative entropy.
For our subsequent convergence result, however, there will be a strong 
difference between the first and the second slot of the relative entropy.
In particular, in our convergence results, Corollaries \ref{cor:super} and
\ref{cor:sub}, we will require the function in the second slot of the relative 
entropy 
to have values only in some compact subset of the set of realizable vectors.

\begin{lemma}
 Let $\bu$ be a Lipschitz continuous solution of \eqref{eq:moment-eq} and let 
$\ug$ be an entropy solution of 
\eqref{eq:reg-moment-eq}, i.e. $\ug \in L^\infty([0,T) \times X, \R^{N+1})$ satisfies 
 \begin{align*}
\partial_t \ug + \nabla_x \cdot \bff_\gamma(\ug) &= \br_\gamma(\ug),\\
 \partial_t \hg(\ug) + \nabla_x \cdot \jg(\ug)
  &\leq \hg'(\ug) \cdot \br_\gamma(\ug)
\end{align*}
in the sense of distributions.
Then, for almost all $0 \leq t \leq T$ the following 
inequality holds:
 \begin{multline}\label{eq:relent}
  \int_X\hg(\ug(t,x) | \bu(t,x)) \, dx  \leq 
   - \int_0^t \int_X (\nabla_x \alphahatgu(s,x)) : \fg(\ug(s,x)| 
\bu(s,x))\\ - q_\gamma(\ug(s,x), \bu(s,x)) 
  - J_\gamma(\ug(s,x), \bu(s,x)) : \nabla_x \bu(s,x) dx ds
 \end{multline}
with
\begin{align}
 \fg(\ug| \bu) := \fg(\ug) - \fg(\bu) - \fg'(\bu)(\ug - \bu)
\end{align}
and
\begin{align}
 q_\gamma(\ug, \bu) := (\hg'(\ug) - \hg'(\bu)) \cdot \rg(\ug)
  - \br(\bu) \cdot (\hg''(\bu)(\ug - \bu))
\end{align}
and
\begin{equation}
 J_\gamma (\ug,\bu):=(\ug - \bu) \hg''(\bu) (\bff'(\bu) - \fg'(\bu)) 
\end{equation}

\end{lemma}

\begin{proof}
The proof follows the proof of \cite[Thm 5.2.1]{Dafermos2016}. We need, 
however, to account for the difference of flux functions in 
\eqref{eq:reg-moment-eq} and \eqref{eq:moment-eq}. Therefore, we provide a brief 
proof that is to be understood in the dual space of
$W^{1,\infty}_0([0,T) \times X, [0,\infty))$, i.e. nonnegative Lipschitz 
continuous functions with compact support. 
\begin{subequations}
\begin{align}
 &\partial_t \hg(\ug | \bu) + \nabla_x \cdot \jg(\ug | \bu) \nonumber \\
 &= \partial_t \hg(\ug) + \nabla_x \cdot \jg(\ug)
  - h_\gamma'(\bu) \cdot \partial_t \bu - \jg'(\bu) : \nabla_x \bu
  - (\hg''(\bu) \partial_t \bu ) \cdot (\ug - \bu) \nonumber \\
 &\qquad - (\hg''(\bu) \cdot \nabla_x \bu ) \cdot (\fg(\ug) - \fg(\bu))
  - \alphahatgu \cdot (\partial_t \ug + \nabla_x\cdot \fg(\ug))
  \nonumber \\
 &\qquad + \alphahatgu \cdot (\partial_t \bu + \fg'(\bu) : \nabla_x \bu) \\
 &\leq \hg'(\ug) r_\gamma(\ug)  
 - \hg''(\bu)(- \bff'(\bu) : \nabla_x \bu + \br(\bu))\cdot (\ug - \bu)
 \nonumber \\
 &
 \qquad- (\hg''(\bu) \cdot \nabla_x \bu ) \cdot (\fg(\ug) - \fg(\bu))
 -\hg'(\bu) \br_\gamma(\ug)
 \label{eq:step-ineq}
 \\
 &=  
 -(\nabla_x \alphahatgu) : \fg(\ug | \bu) + q_\gamma(\ug, \bu) 
  + J_\gamma(\ug, \bu) : \nabla_x \bu,
\end{align}
\end{subequations}
where in \eqref{eq:step-ineq} we have used
\begin{subequations}
\begin{align}
&\partial_t \hg(\ug) + \nabla_x \cdot \jg(\ug) \leq  \hg'(\ug) \cdot 
 \br_\gamma(\ug), \\
 & -h_\gamma'(\bu) \cdot \partial_t \bu - \jg'(\bu) : \nabla_x \bu + 
 \alphahatgu \cdot (\partial_t \bu + \fg'(\bu) :\nabla_x \bu)=0, \\
 & \partial_t \bu = - \bff'(\bu) : \nabla_x \bu + \br(\bu), \text{ and} \\
 & \alphahatgu \cdot (\partial_t \ug + \nabla_x \cdot \fg(\ug)) = 
 \hg'(\bu) \br_\gamma(\ug).
\end{align}
\end{subequations}
We have also used the commutation property
\[ \hg'' \fg' = (\fg')^T \hg''\]
that follows by taking the derivative of the compatibility relation $\hg' \fg'  
= j_\gamma'$.



Now we fix $ 0 \leq t \leq T$ and for $\varepsilon>0$ we test 
\begin{equation}
 \partial_t \hg(\ug | \bu) + \nabla_x \cdot \jg(\ug | \bu) \nonumber
 \leq -(\nabla_x \alphahatgu) : \fg(\ug | \bu) + q_\gamma(\ug, \bu) 
  + J_\gamma(\ug, \bu) : \nabla_x\bu,
\end{equation}
with $ \psi_\varepsilon \in W_0^{1,\infty}([0,T) \times X, [0,\infty))$
\[ \psi_\varepsilon(s,x) := \left\{ \begin{array}{ccc}  1 & : &s < t \\
1 - \frac{s-t}{\varepsilon} & : & t < s < t+\varepsilon\\
0 &:& s > t+ \varepsilon
\end{array}  \right.\]
Sending $\varepsilon$ to zero, we obtain \eqref{eq:relent} for all $t$ that are 
Lebesgue points of the map
$[0,T) \rightarrow \R$, ${t \mapsto \int_X \hg(\ug(t,x)|\bu(t,x)) \, dx}$.
Note that $\int_X \hg(\ug(0,x) | \bu(0,x)) \, dx=0$ since $\bu$, $\ug$ satisfy
the same initial condition.
\end{proof}

At this point, with relative-entropy methods it is typical to look to bound the 
integrand of the right-hand side using $\hg(\ug | \bu)$ so that  Gr\"onwall's 
inequality can be used.

It turns out that such bounds can be obtained if we make some assumption on 
the original solution $\bu$ which bounds it away from the boundary of the 
realizable set $\cR$.
The exact form this assumption takes depends on the entropy; we
will make this precise in Section \ref{sec:estimates}.

\begin{thm}\label{thm:main}
Let there be a subset ${\mathcal{S}}$ of the set of realizable vectors $\cR$ 
such that there exists $\gamma_0 >0$ and constants $C_J,C_{\bff},C_q>0$ and
$D_J,D_{\bff},D_q>0$ so that
\begin{equation}\label{eq:fs-bnd}
 \aligned
 \|J_\gamma(\ug, \bu)\| &\le C_{J} \hg(\ug | \bu) + D_{J} \gamma^2
   \\
 \|\fg(\ug | \bu)\| &\le C_{\bff} \hg(\ug | \bu) + D_{\bff} \gamma^2
   \\
 q_\gamma(\ug, \bu) &\le C_q \hg(\ug | \bu) + D_q \gamma^2 
 \endaligned
 \quad \forall \ug \in \R^{N+1}, \bu \in {\mathcal{S}}, \gamma \in 
(0,\gamma_0).
 \end{equation}
Let $\bu$ be a Lipschitz continuous solution of \eqref{eq:moment-eq} satisfying 
$\bu(t,x) \in {\mathcal{S}}$ for all
 $(t, x) \in [0, T] \times X$. Let $\{\ug\}_{\gamma \in (0,\gamma_0)}$ be a 
family of entropy solutions of 
\eqref{eq:reg-moment-eq}.
Furthermore assume that
${\|\nabla_x \bu\|_{L^\infty([0,T] \times X)} \leq C_u}$,
$ \|\nabla_x \alphahatgu\|_{L^\infty([0,T] \times X)} 
\le C_{\alphahat}$ uniformly in $\gamma$.
Then for $\gamma$ sufficiently small,
\begin{align}
 \int_X \hg(\ug(T, x) | \bu(T, x)) \intdx \le \exp(C T) D T \gamma^2,
\end{align}
where $C := C_{\alphahat} C_{\bff}+ C_u C_J + C_q$ and
${D := C_{\alphahat} D_{\bff}+ C_u D_J + D_q}$.
\end{thm}
\begin{proof}
 Inserting \eqref{eq:fs-bnd} into \eqref{eq:relent} implies
 \begin{multline}
  \int_X\hg(\ug(t,x) | \bu(t,x)) \, dx 
   \\
   \leq \int_0^t \int_X (C_{\alphahat} C_{\bF}+ C_u C_J + C_q) \hg(\ug | \bu)
   +(C_{\alphahat} D_{\bff}+ C_u D_J + D_q) \gamma^2  dx ds.
 \end{multline}
 The assertion of the Theorem follows by applying Grönwall's lemma.
\end{proof}


\begin{remark}
It is important to note that we need to bound the norms of $\fg(\ug | \bu)$ 
and $J_\gamma(\ug, \bu)$ while it is sufficient to bound $q_\gamma(\ug , \bu)$ 
from above (no bound from below is needed).
This is due to the fact that, in \eqref{eq:relent}, $\fg(\ug | \bu)$  and 
$J_\gamma(\ug, \bu)$  are multiplied by $\nabla_x \alphahatgu$ and
$\nabla_x \bu$, respectively,  whose directions are unknown.
\end{remark}

\begin{remark}
We will work with sets $\cS$ such that the derivative
$\alphahatg'(\bu) = ((\hg)''_*(\alphahatgu))^{-1}$ is 
uniformly bounded for $\bu \in \cS$ and $\gamma \in (0, \gamma_0)$.
Thus $\nabla_x \alphahatgu \in L^\infty([0,T] \times X)$ follows from
$\nabla_x \bu \in L^\infty([0,T] \times X)$ and the chain rule, and
$\|\nabla_x \alphahatgu\|_{L^\infty([0,T] \times X)}$ is bounded as
$\gamma \to 0$.
\end{remark}

\section{Estimates}\label{sec:estimates}

In this section we give sets $\cS$ over which the flux and source terms can be 
bounded as required in \thmref{main}.

First, we give some basic properties of realizable moment vectors and the 
functions $\uhat$, $\bff$, and their regularized counterparts which will be 
used repeatedly when estimating the flux and source terms.

\subsection{Basic properties}\label{sec:basic}

For basis functions satisfying \assref{basis-functions}, the components of any 
realizable moment vector $\bu \in \cR$ satisfy
\begin{align}
 |u_i| = |\vint{m_i g}| \le \vint{g} = u_0,
  \qquad
 \text{for all } i \in \{0, 1, \ldots , N \},
\end{align}
thus there exists a constant $C_0 \in (0, \infty)$ such that
\begin{align}\label{eq:C0}
 \|\bu\| \le C_0 u_0 \qquad \text{for all } \bu \in \cR.
\end{align}

Under \assref{v-bnd} there is a $C_1 \in (0, \infty)$ such that
\begin{align}\label{eq:f-jacobian-bnd}
 \|\bff'(\bu)\| \le C_1
\end{align}
for all $\bu \in \cR$ \cite[Lemma 3.1]{AlldredgeSchneider2014}, so $\bff$ is 
globally Lipschitz continuous on $\overline{\cR}$.
The same bound holds globally for the regularized flux, i.e.,
\begin{align}\label{eq:fg-jacobian-bnd}
 \|\fg'(\bu)\| \le C_1
\end{align}
for all $\bu \in \R^{N + 1}$ (this also follows from the argument used in 
\cite[Lemma 3.1]{AlldredgeSchneider2014}).
Since $\bff(0) = 0$ (in the limit sense, since vector $0$ lies on the 
boundary of $\cR$), we have
\begin{align}\label{eq:f-u-bnd}
 \|\bff(\bu)\| \le C_1 \|\bu\|.
\end{align}

The map $\uhat \circ \alphahatg$ plays an important role in our analysis.
This map returns the realizable moment vector that the regularization uses to 
compute the flux and source terms.
It is globally Lipschitz continuous.%
\footnote{
Indeed, the singular values of its Jacobian are bounded by one:
Let $H := \vint{\bm \bm^T \etad''(\alphahatgu \cdot \bm)}$, and recall that it 
is symmetric positive definite.
Then $(\uhat \circ \alphahatg)'(\bu) = H (H + \gamma I)^{-1}$.
For any eigenvalue-eigenvector pair $(\lambda, \bc)$ of $H$, we have
\begin{align}
 (H + \gamma I)^{-1} H^2 (H + \gamma I)^{-1} \bc
  = \left(\frac{\lambda}{\lambda + \gamma} \right)^2 \bc,
\end{align}
from which we concluded that the singular values have the form
$\lambda / (\lambda + \gamma)$ and thus are bounded by one.
}
If we let $C_2$ be its Lipschitz constant, then in particular
\begin{align}\label{eq:C2}
 \|\uhat(\alphahatg(\bw))\| \le C_2 \|\bw\| + \|\uhat(\alphahatg(0))\|
\end{align}%
In all cases we consider, $\lim_{\gamma \to 0} \uhat(\alphahatg(0)) = 0$ (see 
Appendix~\ref{sec:uag0}), so we define
\begin{align}\label{eq:C3}
 C_3 := \sup_{\gamma \in (0, \gamma_0)} \|\uhat(\alphahatg(0))\|
\end{align}
and generally use
\begin{align}\label{eq:uhatag-bnd}
 \|\uhat(\alphahatg(\bw))\| \le C_2 \|\bw\| + C_3
\end{align}
for appropriate values of $\gamma$.
Furthermore, since $\fg = \bff \circ \uhat \circ \alphahatg$,%
\footnote{To see this, start from \eqref{eq:f-and-r-gam} and apply the fact 
that $\alphahat$ is the inverse function of $\uhat$ to get
\begin{align}
 \fg(\bu) = \Vint{v \bm \Ggu} = \Vint{v \bm G_{\alphahat(\uhat(\alphahatgu))}}
  = \bff(\uhat(\alphahatgu));
\end{align}
here we have used $\alphahat \circ \uhat = \id$, and the last equality 
comes from the definition of $\bff(\bu)$ in 
\eqref{eq:f-and-r}.
}
we can combine this bound with \eqref{eq:f-u-bnd} to get
\begin{align}\label{eq:fg-bnd}
 \|\fg(\bw)\| \le C_1 \|\uhat(\alphahatg(\bw))\|
  \le C_1 (C_2 \|\bw\| + C_3).
\end{align}

We conclude this section by recalling a handy property of the regularized 
problem.
The partial derivative of $\alphahatg$ with respect to $\gamma$ was computed in 
\cite[Thm 3]{AllFraHau19}:
\begin{align}\label{eq:ag-dec-gamma}
 \frac{\partial}{\partial \gamma} \alphahatg(\bw) =
  - \hg''(\bw) \alphahatg(\bw).
\end{align}
From the positive-definiteness of $\hg''$ we can immediately conclude that
$\|\alphahatg(\bw)\|^2$ is a decreasing function of $\gamma$ for any
$\bw \in \R^{N + 1}$.
We can further conclude that $\|\alphahatg(\bw)\| \to 0$ as
$\gamma \to \infty$:
Since $\|\alphahatg(\bw)\|$ is a decreasing function of $\gamma$, it remains 
bounded for $\gamma \in (\gamma_0, \infty)$, where 
$\gamma_0 \in (0, \infty)$.
Thus the moment vector $\uhat(\alphahatg(\bw))$ associated with the ansatz
$G_{\alphahatg(\bw)}$ is also bounded for $\gamma \in (\gamma_0, \infty)$,
since $\uhat$ is a continuous map.
Let
\begin{align}
 u_{\gamma_0,\bw} := \sup_{\gamma \in (\gamma_0, \infty)}
  \|\uhat(\alphahatg(\bw))\|.
\end{align}
Then by rearranging the first-order necessary conditions
\eqref{eq:1st-ord-necc} we have
\begin{align}\label{eq:alphag-to-0}
 \|\alphahatg(\bw)\| = \frac{\|\bw - \uhat(\alphahatg(\bw))\|}{\gamma}
  \le \frac{\|\bw\| + u_{\gamma_0,\bw}}{\gamma},
\end{align}
from which we conclude $\|\alphahatg(\bw)\| \to 0$ as $\gamma \to \infty$.

We can also use \eqref{eq:ag-dec-gamma} to show that the entropy $\hg$ is a 
decreasing function of $\gamma$.
Using the formula
$\hg(\bw) = h(\uhat(\alphahatg(\bw))) + \frac \gamma 2 \|\alphahatg(\bw)\|^2$
(obtained by inserting \eqref{eq:entropy-entropy-flux} and 
\eqref{eq:1st-ord-necc} into \eqref{eq:hg-jg}), we have, after some basic 
manipulations,
\begin{align}\label{eq:hg-dec-gamma}
 \frac{\partial}{\partial \gamma} \hg(\bw)
  = -\frac12 \|\alphahatg(\bw)\|^2 \le 0.
\end{align}

\subsection{Two classes of kinetic entropy functions}

To prove the estimates \eqref{eq:fs-bnd}, we need to assume that the true 
solution $\bu$ is bounded away from the boundary of realizability.
This can be achieved in various ways, but the specific form of the assumption 
must depend on properties of the kinetic entropy function $\eta$.

We consider two kinds of kinetic entropy functions.
The first is defined with the Maxwell--Boltzmann entropy,
\begin{align}\label{eq:mb}
 \eta(z) = z \log(z) - z,
\end{align}
in mind.
(The second term is purely for mathematical convenience.)

\begin{definition}
Let $\eta : (0, \infty) \to \R$ satisfy \assref{smooth-eta}.
We call $\eta$ \emph{superlinear} if ${\range(\eta') = \dom(\eta_*) = \R}$.
\end{definition}

Since $\eta'$ is an increasing function, these entropies grow superlinearly as 
$z \to \infty$.
Furthermore, since $0 \in \range(\eta')$, the entropy $\eta$ has a global 
minimum.

\begin{remark}
Note that we use the term \emph{superlinear} even though the functions
$\eta(z) = z^\alpha$ for ${\alpha \in (1, \infty)}$ do not belong to our class 
of superlinear entropy functions.
\end{remark}

The second kind of entropy we consider includes the Bose--Einstein entropy,
\begin{align}
 \eta(z) = z \log(z) - (1 + z) \log (1 + z),
\end{align}
and the Burg entropy,
\begin{align}
 \eta(z) = -\log(z).
\end{align}

\begin{definition}
Let $\eta : (0, \infty) \to \R$ satisfy \assref{smooth-eta}.
We call $\eta$ \emph{sublinear} if 
${\range(\eta') = \dom(\eta_*) = (-\infty, 0)}$ with
$\lim_{z \to \infty} \eta'(z) = 0$ and
$\lim_{z \to 0} \eta'(z) = -\infty$.
\end{definition}

These entropies are monotonically decreasing functions with no global minimum.
The decay as $z \to \infty$ is sublinear.


\subsection{The superlinear case}

For superlinear entropies we consider the family of sets
\begin{align}\label{eq:ass-M}
 \cR^M := \left\{ \uhat(\bsalpha) : \|\bsalpha\| \le M \right\},
\end{align}
for $M \in (0, \infty)$.
For each $M$, we have $\cR^M \subset \subset \cR$, and as $M \to \infty$, the 
set $\cR^M$ approaches the full realizable set $\cR$ (since under
\assref{v-bnd} we have $\cR = \uhat(\R^{N + 1})$).
We will show that the assumptions of \thmref{main} hold for superlinear 
entropies when $\cS = \cR^M$ for any $M$.

First we give some properties on $\cR^M$.
Since $\cR^M$ is a compact set, both
\begin{align}\label{eq:uM}
 u_M := \sup_{\bw \in \cR^M} \|\bw\|
  \qquand
 h_M := \sup_{\bw \in \cR^M} h(\bw)
\end{align}
are finite.
Since $\hg$ is a decreasing function of $\gamma$ (recall 
\eqref{eq:hg-dec-gamma}), we have
\begin{align}\label{eq:hM-uM}
 h_M = \sup_{\substack{\bw \in \cR^M \\ \gamma \in (0, \infty)}} \hg(\bw).
\end{align}

Another important consequence of restricting $\bu$ to $\cR^M$ is that $\hg''$ 
is bounded from above and below over $\cR^M$.
First recall that $\hg'' = ((\hg)''_* \circ \alphahatg)^{-1}$, so we work with
$(\hg)''_*$.
Let $\bc$ be the unit-length eigenvector associated with the largest eigenvalue 
of $(\hg)''_*(\alphahatgu)$ for some $\bu \in \cR^M$.
Then we have
\begin{subequations}
\begin{align}
 \lambda_{\max}((\hg)''_*(\alphahatgu)) &= \bc \cdot \left(
   \Vint{\bm \bm^T \etad''(\alphahatgu \cdot \bm)} + \gamma I \right) \bc \\
  &= \Vint{(\bc \cdot \bm)^2
   \etad''(\alphahatgu \cdot \bm)} + \gamma \\
  &\le |V| \left( \sup_{y \in [-M, M]} \etad''(y) \right) + \gamma_0
\end{align}
\end{subequations}
for $\gamma \le \gamma_0$, where we have used that $\|\alphahatg(\bw)\|$ is a 
decreasing function of $\gamma$ (recall \eqref{eq:ag-dec-gamma}).
Note that, on $\cR^M$, as $\gamma \to \infty$, all eigenvalues of $\hg''$ go to 
zero, and indeed the function $\hg$ becomes flat.
Since the behavior of the regularized equations for large $\gamma$ is not 
particularly interesting, we rule out these problems by considering only 
$\gamma$ smaller than some arbitrary $\gamma_0$.

On the other hand, if we now let $\bc$ be the unit-length eigenvector 
associated with the smallest eigenvalue of $(\hg)''_*(\alphahatgu)$ for
$\bu \in \cR^M$ we have
\begin{align}
 \lambda_{\min}((\hg)''_*(\alphahatgu))
  = \Vint{(\bc \cdot \bm)^2 \etad''(\alphahatgu \cdot \bm)} + \gamma
  \ge \lambda_{\min}(\vint{\bm \bm^T}) \inf_{y \in [-M, M]} \etad''(y)
\end{align}
(where $\lambda_{\min}(\vint{\bm \bm^T}) > 0$ because $\bm$ span a basis).
Note that \assref{smooth-eta} guarantees strict positivity of $\etad''$ because
$\etad''(y) = 1 / \eta''(\etad'(y))$.

Thus for $\hg''$ we can conclude the existence of positive constants
$\lambda_{\min, h'', M}$
and $C_{h'',  M}$ such that
\begin{align}\label{eq:hess-bnds}
 \aligned
  \bv \cdot \hg''(\bu)\bv &\ge \lambda_{\min, h'', M} \|\bv\|^2 \\
  \|\hg''(\bu)\| &\le C_{h'', M}
 \endaligned
 \quad
 \text{for all } \bv \in \R^{N + 1} \text{, } \bu \in \cR^M
  \text{, and } \gamma \in (0, \gamma_0)
\end{align}

\paragraph{The flux terms}

\begin{lemma}\label{lem:flux-bnd-super}
Let $\eta$ be a superlinear kinetic entropy function, $M \in (0, \infty)$, 
and $\gamma_0 \in (0, \infty)$.
Then there exist positive constants $C_{\bff}$, $C_J$, and $D_J$ such that
\begin{equation}\label{eq:flux-bnd-lem-super}
 \aligned
 \|\fg(\ug | \bu)\| &\le C_{\bff} \hg(\ug | \bu)
   \\
 \|J_\gamma(\ug, \bu)\| &\le C_{J} \hg(\ug | \bu) + D_{J} \gamma^2
 \endaligned
 \quad \forall \ug \in \R^{N+1}, \bu \in \cR^M, \gamma \in (0,\gamma_0).
\end{equation}
These constants depend on $M$ and $\gamma_0$ but are independent of $\ug$, 
$\bu$, and $\gamma$.
\end{lemma}

In the proof of \lemref{flux-bnd-super} we often use the following elementary 
lemma.

\begin{lemma}\label{lem:fz}
Let $f : [0, \infty) \to \R$ be a twice continuously differentiable function 
such that $f''(z) > 0$ for all $z$, and let $a \in (0, \infty)$, $b \in \R$, 
and $c \in \R$ be given.
If $K \in (0, \infty)$ and $C \in (0, \infty)$ satisfy
\begin{align}\label{eq:prototype-conditions}
 f(K) + c > 0, \quad f'(K) > 0, \quand
 C \ge \max\left\{ \frac{aK + b}{f(K) + c}, \frac{a}{f'(K)} \right\},
\end{align}
then
\begin{align}\label{eq:prototype}
 a z + b \le C(f(z) + c)
\end{align}
holds for all $z \ge K$.
\end{lemma}

In the proof of \lemref{flux-bnd-super} we also use the following lemma to get 
a nonzero lower bound the relative entropy.

\begin{lemma}\label{lem:C-M-L-positive}
Let $M \in (0, \infty)$, $L \in (M, \infty)$, and $\gamma_0 \in (0, \infty)$.
Then
\begin{align}\label{eq:C-M-L}
 C_{h, M, L} := \inf_{\substack{\bv \in \R^{N + 1} \setminus \cR^L \\
                                \bu \in \cR^M \\
                                \gamma \in (0, \gamma_0)}}
                 \hg(\bv | \bu)
\end{align}
is strictly positive.
\end{lemma}

The proof of \lemref{C-M-L-positive} can be found in 
Appendix~\ref{sec:C-M-L-positive}

\begin{proof}[\lemref{flux-bnd-super}]

We partition $\R^{N + 1}$ into three subsets and consider $\ug$ on each of these 
three subsets, which we illustrate in \figref{sets}.
On the first set, $\cR^L$ for $L \in (M, \infty)$, we take advantage of the 
fact that $\fg$, $J_\gamma$, and the relative entropy all look like
$\|\ug - \bu\|^2$ locally, up to an $\cO(\gamma^2)$ term.
On the second set, $B_K \setminus \cR^L$, where $B_K$ is a norm ball in
$\R^{N + 1}$, we use the fact that neither $\fg$ nor $J_\gamma$ nor $q_\gamma$ 
blow up on compact sets.
Finally, in the third set we use that $\fg$ and $J_\gamma$ grow linearly in 
either $\|\ug\|$ or $\u0hat(\alphahatg(\ug))$ for large $\|\ug\|$ 
while $\hg(\ug)$ grows at least linearly in $\|\ug\|$ or 
$\u0hat(\alphahatg(\ug))$.

\begin{figure}
\large
\input{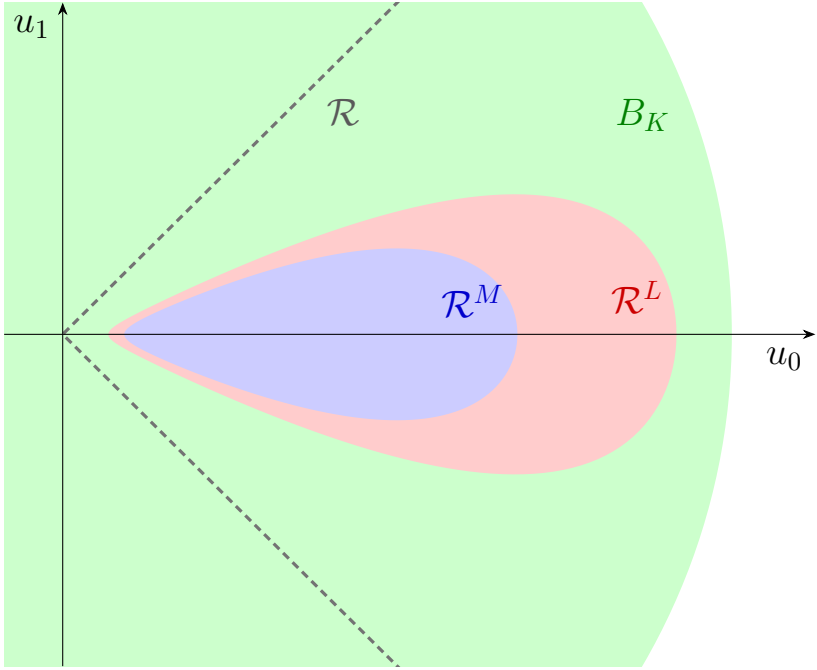}
\caption{The sets used in the proof of \lemref{flux-bnd-super}.
For this figure, we consider the $M_1$ case in slab geometry: $V = [-1, 1]$,
$\bm(v) = (1, v)$.
In this case, the realizable set is given by
$\cR = \{ (u_0, u_1) : |u_1| < u_0\}$, and we used $M = 1$, $L = 1.3$, and
$K = 8$.}
\label{fig:sets}
\end{figure}

\begin{enumerate}[(i)]
 \item We begin with $(\ug, \bu) \in \cR^L \times \cR^M$ for an
  $L \in (M, \infty)$.
  (The reason for choosing $L > M$ is given in case (ii) below.)
  
  When $\fg$ is sufficiently smooth there exists an $\bw$ on the line
  connecting $\ug$ and $\bu$ such that
  \begin{align}
   \fg(\ug | \bu) = (\fg''(\bw) (\ug - \bu))(\ug - \bu)\,.
  \end{align}
  In our case, $\fg$ indeed possesses the requisite smoothness:
  We can write $\fg$ as ${\fg = \bg \circ \alphahatg}$, where
  $\bg(\bsalpha) = \Vint{v \bm \Ga}$.
  The function $\bg$ is smooth and its derivatives are bounded over any bounded 
  set.
  Note that $\bw \in \conv(\cR^L)$, and since $\cR^L \subset \subset \cR$ and
  $\cR$ is convex, we also know $\conv(\cR^L) \subset \subset \cR$.
  Thus we define
  \begin{align}\label{eq:L-tilde}
   \widetilde L := \sup_{\bw \in \conv(\cR^L)} \|\alphahat(\bw)\| < \infty.
  \end{align}
  Since $\alphahatg(\bw)$ is continuous with respect to $\gamma$ for
  $\gamma \in [0, \infty)$ (where $\alphahat_{\gamma = 0} = \alphahat$ when 
  $\bw \in \cR$) \cite[\S 3.1]{AllFraHau19} and $\|\alphahatg(\bw)\|$ 
  is a decreasing function of $\gamma$ (recall \eqref{eq:ag-dec-gamma}) we 
  have
  \begin{align}
   \widetilde L = \sup_{\substack{\bw \in \conv(\cR^L) \\
                                   \gamma \in (0, \infty)}}
    \|\alphahatg(\bw)\|.
  \end{align}
  
  Thus it is clear that we only consider $\bg$ and its derivatives within a
  bounded set.
  Furthermore, when bounded away from the boundary of $\cR$, $\alphahatgu$ is a 
  smooth function of $\bu$, and this smoothness is uniform for
  $\gamma \in (0, \gamma_0)$.
  For example, $\alphahatg' = \hg''$ (recall \eqref{eq:hg-hess}), and from
  \eqref{eq:hess-bnds} we have
  $\|\hg''(\bw)\| \le C_{h'', \widetilde L}$ for
  $\bw \in \cR^L$.
  The second derivative $\alphahatg''$ can be similarly bounded, but we omit 
  this more tedious computation.
  
  So we define
  \begin{align}\label{eq:C-fg''}
   C_{\bff'', L}
    := \sup_{\substack{\bw \in \conv(\cR^L) \\
             \gamma \in (0, \gamma_0)}}
        \|\fg''(\bw)\|
  \end{align}
  for some $\gamma_0 \in (0, \infty)$ and conclude
  \begin{align}\label{eq:fg-i}
   \| \fg(\ug | \bu) \| \le C_{\bff'', L} \|\ug - \bu\|^2
  \end{align}
  for $(\ug, \bu) \in \cR^L \times \cR^M$.
  
  We now turn to $J_\gamma$.
  The deciding factor in $J_\gamma$ is $(\bff'(\bu) - \fg'(\bu))$; to
  estimate it we use $\fg = \bff \circ \uhat \circ \alphahatg$, the 
  Lipschitz continuity of $\bff'$ on $\cR^M$ (guaranteed by
  \eqref{eq:C-fg''}), and the accuracy inequality
  \begin{align}
   \|\uhat(\alphahatgu) - \bu\| \le M \gamma
  \end{align}
  from \cite[Thm. 2]{AllFraHau19} as follows:
  \begin{subequations}
  \label{eq:lip-fg'}
  \begin{align}
   \|\bff'(\bu) - \fg'(\bu)\| &= \|\bff'(\bu)
     - \bff'(\uhat(\alphahatgu))(\uhat \circ \alphahatg)'(\bu)\| \\
    &= \|\bff'(\bu) - \bff'(\bu)(\uhat \circ \alphahatg)'(\bu) \nonumber \\
    &\qquad + \bff'(\bu)(\uhat \circ \alphahatg)'(\bu)
     - \bff'(\uhat(\alphahatgu))(\uhat \circ \alphahatg)'(\bu)\| \\
    &\le \|\bff'(\bu) (I - (\uhat \circ \alphahatg)'(\bu))\| \nonumber \\
    &\qquad + \|(\bff'(\bu)
         - \bff'(\uhat(\alphahatgu)))(\uhat \circ \alphahatg)'(\bu)\| \\
    &\le C_1 \gamma \|\hg''(\bu)\|
     + C_{\bff'', M}
       M \gamma \|(\uhat \circ \alphahatg)'(\bu)\| \label{eq:J-fg-hess-step} \\
    &\le (C_1 C_{h'', M} + C_{\bff'', M} M C_2) \gamma \\
    &=: C_4 \gamma,
  \end{align}
  \end{subequations}
  where we define $C_4 := C_1 C_{h'', M} + C_{\bff'', M} M C_2$.
  In \eqref{eq:J-fg-hess-step} we have used
  \begin{align}
   I - (\uhat \circ \alphahatg)'(\bu) = \gamma \hg''(\bu),
  \end{align}
  which is a straightforward computation using $\uhat' = h''_*$ and   
  ${\alphahatg'(\bu)= (h''_*(\alphahatgu) + \gamma I)^{-1}}$ (see   
  \eqref{eq:hg-hess}).
  In this step we are also able to use $C_{\bff'', M}$ because
  $\|\alphahatg(\bu)\|$ is a decreasing function of $\gamma$ (recall 
  \eqref{eq:ag-dec-gamma}).

  With \eqref{eq:lip-fg'} and the bound on $\hg''$ from \eqref{eq:hess-bnds} 
  we can bound $J_\gamma$ by
  \begin{subequations}
  \label{eq:J-i}
  \begin{align}
   \|J_\gamma(\ug, \bu)\| &\le \|\ug - \bu\| \|\hg''(\bu)\|  
     \|\bff'(\bu) - \fg'(\bu)\| \\
    &\le \|\ug - \bu\| C_{h'', M} C_4 \gamma \\
    &\le \frac{C_{h'', M} C_4}2 \left( \gamma^2 + \|\ug - \bu\|^2 
                                       \right),
  \end{align}
  \end{subequations}
  where for the last step we have applied Young's inequality.
  
  Finally, we must bound $\hg(\ug | \bu)$ similarly from below.
  This follows immediately from \eqref{eq:hess-bnds}:
  \begin{align}\label{eq:rel-ent-bnd-below-L}
   \hg(\ug | \bu) = (\ug - \bu) \cdot \hg''(\bw)(\ug - \bu)
    \ge \lambda_{\min, h'', \widetilde L} \|\ug - \bu\|^2
  \end{align}
  for all $(\ug, \bu) \in \cR^L \times \cR^M$ and $\bw$ is the appropriate 
  vector in $\conv(\cR^L)$ from the mean-value theorem.
  
  Altogether we have
  \begin{align}
   \|\fg(\ug | \bu)\| &\le \frac{C_{\bff'', L}}
                                {\lambda_{\min, h'', \widetilde L}} 
    \hg(\ug | \bu) \\
   \|J_\gamma(\ug, \bu)\| &\le \frac{C_{h'', M} C_4}
                                    {2 \lambda_{\min, h'', \widetilde L}}
                               \hg(\ug | \bu)
                               + \frac{C_{h'', M} C_4}2 \gamma^2
  \end{align}
  for $(\ug, \bu) \in \cR^L \times \cR^M$.

 \item We now consider $\ug \in B_K \setminus \cR^L$ for the ball
  $B_K := \{\bw \in \R^{N + 1} : \|\bw\| \le K\}$.
  Since $B_K$ is a compact set, we know that the constants
  \begin{align}\label{eq:fg-Jg-B_K}
   C_{\bff, K} &:= \sup_{\substack{\ug \in B_K \setminus \cR^L \\ 
                                        \bu \in \cR^M \\
                                        \gamma \in (0, \gamma_0)}}
                         \|\fg(\ug | \bu)\| \\
   C_{J, K} &:= \sup_{\substack{
                                    \ug \in B_K \setminus \cR^L \\ 
                                    \bu \in \cR^M \\
                                    \gamma \in (0, \gamma_0)}}
                         \|J_\gamma(\ug, \bu)\|
  \end{align}
  are finite for any $\gamma_0 \in (0, \infty)$.
  Indeed with \eqref{eq:fg-jacobian-bnd}, \eqref{eq:fg-bnd}, \eqref{eq:uM}, and 
  \eqref{eq:hess-bnds}
  we immediately have the crude upper bounds
  \begin{align}
   C_{\bff, K} \le C_1((C_2 + 1) (K + u_M) + 2C_3) \quand
   C_{J, K} \le 2(K + u_M) C_{h'', M} C_1.
  \end{align}
  
  To bound $\hg(\ug | \bu)$  from below we use $C_{h, M, L}$ from 
  \lemref{C-M-L-positive}.
  (Note that here $L > M$ is crucial: this lower bound is not strictly positive 
  for $L \le M$.)
  
  Altogether we have
  \begin{align}
   \|\fg(\ug| \bu)\| \le \frac{C_{\bff, K}}{C_{h, M, L}}
    \hg(\ug | \bu)
   \qquand
   \|J_\gamma(\ug, \bu)\| \le \frac{C_{J, K}}{C_{h, M, L}}
    \hg(\ug | \bu).
  \end{align}
  Up to now, $K$ is arbitrary; in the next and final case we give a lower bound 
  that $K$ must satisfy.
  
 \item Finally, for $\ug \in \R^{N + 1} \setminus B_K$, how we proceed depends 
  on which term in $\hg$ dominates: when $\u0hat(\alphahatg(\ug))$ is large, we 
  use the entropy term, otherwise we use the quadratic term.
  To distinguish these two subcases, let $\delta \in (0, 1)$.
  The following arguments work for any value of $\delta \in (0, 1)$; the 
  particular choice of $\delta$ merely affects the value of the constants
  $C_{\bff}$ and $C_J$.
  \begin{enumerate}[(a)]
   \item $\|\uhat(\alphahatg(\ug))\| \le \delta \|\ug\|$ \\
    Here, we use the quadratic term of the relative entropy to dominate the 
    linear term of the relative flux.
    We bound the relative flux by
    \begin{align}\label{eq:fg-rel-iii-a}
     \|\fg(\ug, \bu)\| &\le C_1 \|\uhat(\alphahatg(\ug))\| + C_1 \|\ug\|
       + C_1 (C_2 u_M + C_3) + C_1 u_M \nonumber \\
      &\le C_1 (1 + \delta) \|\ug\| + C_1 ((C_2 + 1) u_M + C_3),
    \end{align}
    $J_\gamma$ by
    \begin{align}\label{eq:J-iii-a}
     \|J_\gamma(\ug, \bu)\| \le 2 C_1 C_{h'', M} (\|\ug\| + u_M),
    \end{align}
    and the relative entropy by
    \begin{align}
     \hg(\ug | \bu) &\ge h(\uhat(\alphahatg(\ug)))
       + \frac1{2\gamma} (\|\ug\| - \|\uhat(\alphahatg(\ug))\|)^2
       - h_M - M(\|\ug\| + u_M) \nonumber \\
      &\ge |V|\eta_{\min} + \frac{(1 - \delta)^2}{2\gamma_0} \|\ug\|^2
       - h_M - M(\|\ug\| + u_M) \label{eq:hg-lb-iii-a}
    \end{align}
    where we use $\eta_{\min} := \min_{z \ge 0} \eta(z)$, and in the first 
    inequality we have used that $\|\alphahatg(\bu)\|$ is a decreasing function 
    of $\gamma$ (recall \eqref{eq:ag-dec-gamma}).
    
    Then the application of \lemref{fz} with
    $z = \|\ug\|$ and ${f(z) = (1 - \delta)^2 z^2 / (2 \gamma_0) - M z}$ gives 
    the following conditions on $K$, $C_{\bff}$, and $C_J$:
    \begin{gather}
     \frac{(1 - \delta)^2}{\gamma_0} K^2 - M K
      >  h_M + M u_M - 2\eta_{\min},
      \quad
     \frac{(1 - \delta)^2}{\gamma_0} K > M, \label{eq:K-conds-a} \\
     \aligned
      C_{\bff} &\ge \max\left\{
       \frac{C_1 (1 + \delta) K + C_1 ((C_2 + 1) u_M + C_3)}
       {\frac{(1 - \delta)^2}{\gamma_0} K^2 - M K + |V|\eta_{\min} - h_M
        - M u_M},
       \frac{C_1 (1 + \delta)}{\frac{(1 - \delta)^2}{\gamma_0}K - M}
      \right\}, \text{ and} \\
      C_{J} &\ge \max\left\{
       \frac{2 C_1 C_{h'', M} (K + u_M)}
       {\frac{(1 - \delta)^2}{\gamma_0} K^2 - M K + |V|\eta_{\min} - h_M
        - M u_M},
       \frac{2 C_1 C_{h'', M}}{\frac{(1 - \delta)^2}{\gamma_0}K - M}
      \right\}
     \endaligned
    \end{gather}

   \item $\|\uhat(\alphahatg(\ug))\| > \delta \|\ug\|$ \\
    In this case, we know that $\uhat(\alphahatg(\ug))$ is not arbitrarily 
    small and therefore formulate our bounds in terms of
    $\u0hat(\alphahatg(\ug))$.
    First, for the upper bounds we have, using \eqref{eq:C0},
    \begin{align}
     \| \fg(\ug, \bu)\| &\le \left(C_1 C_0 + C_1 \frac{C_0}\delta\right) 
      \u0hat(\alphahatg(\ug)) + C_1 (C_2 u_M + C_3) + C_1 u_M, \text{ and}
      \label{eq:fg-rel-iii-b} \\
     \|J_\gamma(\ug, \bu)\| &\le 2 C_1 C_{h'', M} \left(
      \frac{C_0}\delta \u0hat(\alphahatg(\ug)) + u_M \right).
      \label{eq:J-iii-b}
    \end{align}

    For the lower bound on the entropy, notice that by Jensen's inequality
    \begin{align}
     \Vint{\eta(g)} \ge |V| \eta\left( \frac1{|V|} \Vint{g} \right),
    \end{align}
    from which it follows that
    \begin{align}
     h(\bu) \ge |V| \eta\left( \frac1{|V|} u_0 \right),
    \end{align}
    for any realizable $\bu$.
    Thus for the relative entropy we have
    \begin{align}\label{eq:hg-lb-iii-b}
     \hg(\ug | \bu) \ge |V|\eta\left(\frac1{|V|} \u0hat(\alphahatg(\ug))\right) 
      - h_M
      - M \left( \frac{C_0}{\delta} \u0hat(\alphahatg(\ug)) + u_M 
      \right).
    \end{align}
    
    Now, with the lower bound on $\u0hat(\alphahatg(\ug))$, namely,
    \begin{align}
     \u0hat(\alphahatg(\ug)) \ge \frac1{C_0} \|\uhat(\alphahatg(\ug))\|
      \ge \frac\delta{C_0} \|\ug\|
      \ge \frac{\delta K}{C_0}
    \end{align}
    we can apply \lemref{fz} with $z = \u0hat(\alphahatg(\ug))$ and
    ${f(z) = |V|\eta(z / |V|) - M C_0 z / \delta}$.
    Here the conditions on $K$, $C_{\bff}$, and $C_J$ become
    \begin{gather}
     |V|\eta\left( \frac{\delta K}{|V| C_0} \right) - M K > h_M + M u_M, \quad
      \eta'\left( \frac{\delta K}{|V| C_0} \right) > \frac{M C_0}\delta, 
      \label{eq:K-conds-b} \\
     \aligned
      C_{\bff} &\ge \max\left\{ \frac{C_1((\delta + 1)K + (C_2 + 1)u_M + C_3)}
        {|V|\eta\left( \frac{\delta K}{|V| C_0} \right)
         - M K - h_M - M u_M}, 
        \frac{C_1 C_0 \left( 1 + \frac1\delta \right)}
         {\eta'\left( \frac{\delta K}{|V| C_0} \right) - \frac{M C_0}\delta}
        \right\}, \text{ and} \\
      C_J &\ge \max\left\{ \frac{2 C_1 C_{h'', M} \left(
                                 K + u_M \right)}
        {|V|\eta\left( \frac{\delta K}{|V| C_0} \right)
         - M K - h_M - M u_M}, 
        \frac{2 C_0 C_1 C_{h'', M}}
         {\delta \left( \eta'\left( \frac{\delta K}{|V| C_0} \right)
          - \frac{M C_0}\delta \right)}
        \right\}
     \endaligned
    \end{gather}
  \end{enumerate}
\end{enumerate}

Altogether we get \eqref{eq:flux-bnd-lem-super} for
\begin{align}
 C_{\bff} &= \max\left\{ \frac{C_{\bff'', L}}
                              {\lambda_{\min, h'', \widetilde L}},
  \frac{C_{\bff, K}}{C_{h, M, L}},
  \frac{C_1 (1 + \delta) K + C_1 ((C_2 + 1) u_M + C_3)}
   {\frac{(1 - \delta)^2}{\gamma_0} K^2 - M K + |V|\eta_{\min}
    - h_M - M u_M}, \right. \nonumber \\
 & \qquad \qquad \left.
  \frac{C_1 (1 + \delta)}{\frac{(1 - \delta)^2}{\gamma_0}K - M},
  \frac{C_1((\delta + 1)K + (C_2 + 1)u_M + C_3)}
       {|V|\eta\left( \frac{\delta K}{|V| C_0} \right)
        - M K - h_M - M u_M},
  \frac{C_1 C_0 \left( 1 + \frac1\delta \right)}
       {\eta'\left( \frac{\delta K}{|V| C_0} \right) - \frac{M C_0}\delta} 
  \right\}, \\
 C_J &= \max\left\{ \frac{C_{h'', M} C_4}
                         {2 \lambda_{\min, h'', \widetilde L}},
   \frac{C_{J, K}}{C_{h, M, L}},
   \frac{2 C_1 C_{h'', M} (K + u_M)}
        {\frac{(1 - \delta)^2}{\gamma_0} K^2 - M K + |V|\eta_{\min} - h_M
         - M u_M},
   \right. \nonumber \\
  & \qquad \qquad \left.
   \frac{2 C_1 C_{h'', M}}{\frac{(1 - \delta)^2}{\gamma_0}K - M},
   \frac{2 C_1 C_{h'', M} \left( K + u_M \right)}
        {|V|\eta\left( \frac{\delta K}{|V| C_0} \right)
         - M K - h_M - M u_M}, 
   \frac{2 C_0 C_1 C_{h'', M}}
        {\delta \left( \eta'\left( \frac{\delta K}{|V| C_0} \right)
         - \frac{M C_0}\delta \right)}
   \right\}, \text{ and} \\
 D_J &= \frac{C_{h'', M} C_4} 2,
\end{align}
where $K$ satisfies \eqref{eq:K-conds-a} and \eqref{eq:K-conds-b}.
\end{proof}

\paragraph{The source term}

\begin{lemma}\label{lem:src-bnd-super}
Let $\eta$ be a superlinear kinetic entropy function, $M \in (0, \infty)$, 
and $\gamma_0 \in (0, \infty)$.
Then there exist positive constants $C_q$ and $D_q$ such that
\begin{equation}\label{eq:src-bnd-lem-super}
 q_\gamma(\ug, \bu) \le C_q \hg(\ug | \bu) + D_q \gamma^2
 \quad \forall \ug \in \R^{N+1}, \bu \in \cR^M, \gamma \in (0,\gamma_0).
\end{equation}
\end{lemma}

\begin{proof}
We use the same lower bounds just derived for the 
relative entropy $\hg(\ug | \bu)$ in \lemref{flux-bnd-super}.
Thus we only need to give upper bounds of $q_\gamma$ on the same decomposition 
of $\R^{N + 1} \times \cR^M$ used in the proof of \lemref{flux-bnd-super}.

\begin{enumerate}[(i)]
 \item Let $(\ug, \bu) \in \cR^L \times \cR^M$.
  To write down an upper bound of $q_\gamma$, we define the constants
  \begin{align}\label{eq:h'''-rM}
   C_{h'''} = \sup_{\substack{\bv \in \conv(\cR^L) \\
                               \gamma \in (0, \gamma_0)}}
     \|\hg'''(\bv)\| \qquand
   r_M = \sup_{\bv \in \cR^M}
     \|\br(\bv)\|,
  \end{align}
  all finite by smoothness of $\hg$ and $\br$ and compactness of
  $\conv(\cR^L)$ and $\cR^M$, and we use
  \begin{align}\label{eq:approx-acc-M}
   \|\uhat(\alphahatg(\ug)) - \bu\| \le \|\ug - \bu\| + M\gamma,
  \end{align}
  which follows from $\bu \in \cR^M$ \cite[Thm. 2]{AllFraHau19}.
  Then we rearrange $q_\gamma$ and straightforwardly get the estimate
  \begin{subequations}
  \begin{align}
   q_\gamma(\ug, \bu) &= (\hg'(\ug) - \hg'(\bu)) \cdot (\rg(\ug) - \br(\bu)) 
     \nonumber \\
    & \qquad + (\hg'(\ug) - \hg'(\bu) - \hg''(\bu)(\ug - \bu))
     \cdot \br(\bu)\, \\
    &\le (C_{h'', \widetilde L} C_{\br}
     + r_M C_{h'''})\|\ug - \bu\|^2
     + C_{h'', \widetilde L} C_{\br} M \gamma \|\ug - \bu\| \\
    &\le (C_{h'', \widetilde L} C_{\br} + r_M C_{h'''})\|\ug - \bu\|^2
     + \frac{C_{h'', \widetilde L} C_{\br} M}2
    \left( \gamma^2 + \|\ug - \bu\|^2 \right),
  \end{align}
  \end{subequations}
  which appropriately mirrors \eqref{eq:fg-i} and \eqref{eq:J-i}.
  
 \item Now for the case $\ug \in B_K \setminus \cR^L$ (where again
  $B_K := \{\bw \in \R^n : \|\bw\| \le K\}$), we first use
  $\hg'(\ug) \cdot \rg(\ug) \le 0$ to get
  \begin{align}
   q_\gamma(\ug, \bu) \le -\hg'(\bu) \cdot \rg(\ug)
    - \br(\bu) \cdot (\hg''(\bu)(\ug - \bu))
  \end{align}
  As with the flux, none of these terms blow up for
  $(\ug, \bu) \in B_K \times \cR^M$ for any finite $K$:
  \begin{subequations}
  \begin{align}
   C_{q, K, \gamma_0} &:= \sup_{\substack{\ug \in B_K \setminus \conv(\cR^L) \\ 
                                         \bu \in \cR^M \\
                                         \gamma \in (0, \gamma_0)}}
                         -\hg'(\bu) \cdot \rg(\ug)
                         - \br(\bu) \cdot (\hg''(\bu)(\ug - \bu)) \\
                      &\le M C_{\br} (C_2 K + C_3)
                       + r_M C_{h'', M} (K + u_M).
  \end{align}
  \end{subequations}
  
 \item For large $\ug$ we show, as with the flux, that $q_\gamma$ grows 
  linearly with $\|\ug\|$ when $\uhat(\alphahatg(\ug))$ is small and linearly 
  with $\u0hat(\alphahatg(\ug))$ otherwise.
  Indeed, if for $\delta \in (0, 1)$ we have
  $\|\uhat(\alphahatg(\ug))\| \le \delta \|\ug\|$, then
  \begin{align}\label{eq:q-iii-a}
   q_\gamma(\ug, \bu) \le M C_{\br} \delta \|\ug\|
    + r_M C_{h'', M} (\|\ug\| + u_M),
  \end{align}
  i.e., linear growth in $\|\ug\|$ as in \eqref{eq:fg-rel-iii-a} and
  \eqref{eq:J-iii-a}.
  On the other hand, when ${\|\uhat(\alphahatg(\ug))\| > \delta \|\ug\|}$, 
  then
  \begin{align}\label{eq:q-iii-b}
   q_\gamma(\ug, \bu) \le M C_{\br} C_0 \u0hat(\alphahatg(\ug))
    + r_M C_{h'', M} \left( \frac{C_0}{\delta}\u0hat(\alphahatg(\ug))
     + u_M \right),
  \end{align}
  which is linear growth in $\u0hat(\alphahatg(\ug))$ as in
  \eqref{eq:fg-rel-iii-b} and  \eqref{eq:J-iii-b}.
\end{enumerate}

Now we simply need to change the numerators from the flux case above to get
\eqref{eq:src-bnd-lem-super} for
\begin{align}
 C_q &\ge \max\left\{ \frac{(C_{h'', \widetilde L} C_{\br} + r_M C_{h'''})
                            + \frac{C_{h'', \widetilde L} C_{\br} M}2}
                           {\lambda_{\min, h'', \widetilde L}},
  \frac{C_{q, K, \gamma_0}}{C_{h, M, L}}, \right. \nonumber \\
 & \qquad \qquad \left. \frac{M C_{\br} \delta K + r_M C_{h'', M}
                              (K + u_M)}
   {\frac{(1 - \delta)^2}{\gamma_0} K^2 - M K + |V|\eta_{\min}
    - h_M- C_0 v_{0, M}},
  \frac{M C_{\br} \delta + r_M C_{h'', M}}
       {\frac{(1 - \delta)^2}{\gamma_0}K - M}, \right. \nonumber \\
 & \qquad \qquad \left.
  \frac{M C_{\br} \delta K + r_M C_{h'', M}(K + u_M)}
       {|V|\eta\left( \frac{\delta K}{|V| C_0} \right)
        - M K - h_M - M u_M},
  \frac{M C_{\br} C_0 + r_M C_{h'', M} \frac{C_0}{\delta}}
       {\eta'\left( \frac{\delta K}{|V| C_0} \right) - \frac{M C_0}\delta} 
  \right\}
\end{align}
for some $K$ and $\delta$ satisfying the same conditions as for the flux term, 
and
\begin{align}
 D_q := \frac{C_{h'', \widetilde L} C_{\br} M}2.
\end{align}

\end{proof}

Lemmas~\ref{lem:flux-bnd-super} and \ref{lem:src-bnd-super} yield a more 
precise version of Theorem~\ref{thm:main}:

\begin{cor}\label{cor:super}
Let $\eta$ be a superlinear kinetic entropy function and $\bu$ a Lipschitz  
solution of the entropy-based moment equations \eqref{eq:moment-eq} for which 
there exists $M \in (0, \infty)$ such that $\bu(t, x) \in \cR^M$ for all 
$(t, x) \in [0, T] \times X$.
Let $\{\ug\}_{\gamma\in (0, \gamma_0)} $ be a family of entropy solutions of 
\eqref{eq:reg-moment-eq}.
Furthermore assume that
$\|\nabla_x \bu\|_{L^\infty([0,T] \times X)} \le C_u$. 

Then
${\|\nabla_x \alphahatgu\|_{L^\infty([0,T] \times X)} \le C_{h'' , M} 
C_u}$, and
\begin{align}\label{eq:cor-super}
 \int_X \hg(\ug(T, x) | \bu(T, x)) \intdx \le \exp(C T) D T \gamma^2
\end{align}
for $C := C_{h'' , M} C_u C_{\bff}+ C_u C_J + C_q$ and
$D := C_u D_J + D_q$, where the constants $C_{\bff}$, $C_J$, $C_q$, 
$D_J$, and $D_q$ are given by Lemmas~\ref{lem:flux-bnd-super} and 
\ref{lem:src-bnd-super}.
\end{cor}

%

\begin{cor}\label{cor:super-L2}
Let $\eta$, $\bu$, and $\ug$ satisfy the conditions of 
Corollary~\ref{cor:super} and $L \in (M, \infty)$, and consider
\[ X_L := \{x  \in X \, : \, \ug(T,x) \in \cR^L\}. \]
Then the Lebesgue-measure of the complement of this set satisfies
\begin{align}\label{eq:X-minus-XL-bnd}
 |X \setminus X_L| \le \frac{\exp(C T) D T}{C_{h, M, L}} \gamma^2,
\end{align}
where $C_{h, M, L}$ is defined in \eqref{eq:C-M-L}.
Furthermore on $X_L$ we have
\begin{align}\label{eq:cor-L2}
 \| \ug(T, \cdot) - \bu(T, \cdot) \|_{L^2(X_L)}
  \le \sqrt{\frac{\exp(C T) D T}{\lambda_{\min, h'', \widetilde L}}} 
   \gamma
\end{align}
\end{cor}

\begin{proof}
Recall the definition of $C_{h, M, L}$ in \eqref{eq:C-M-L}.
Then we have
\begin{align}
 C_{h, M, L} |X \setminus X_L|
  \le \int_{X \setminus X_L} \hg(\ug(T, x) | \bu(T, x)) \intdx
  \stackrel{\eqref{eq:cor-super}}{\le} \exp(C T) D T \gamma^2,
\end{align}
from which we conclude \eqref{eq:X-minus-XL-bnd}. 
The inequality \eqref{eq:cor-L2} follows immediately from 
\eqref{eq:rel-ent-bnd-below-L} (cf.\ \eqref{eq:hess-bnds}).
\end{proof}

\subsection{The sublinear case}\label{sec:mods-for-be-entropy}

In contrast to the Maxwell--Boltzmann-like entropies, the entropies $\eta$ we 
consider in this class are not bounded from below but
${\lim_{z \to \infty} \eta'(z) = 0}$.
Consequently ${\dom(\etad) \subseteq (-\infty, 0)}$, and the multipliers must 
satisfy $\bsalpha \cdot \bm < 0$ for all $v \in V$.
For such entropies we replace the assumption in \eqref{eq:ass-M} with
\begin{align}\label{eq:ass-Mm}
 \bu \in \cR^{M, m} := \left\{ \uhat(\bsalpha) : \|\bsalpha\| \le M
                               \text{ and } \bsalpha \cdot \bm(v) \le -m
                               \text{ for all } v \in V
                       \right\},
\end{align}
for some $M \in (0, \infty)$ and $m \in (0, \infty)$.
Note that as $M \to \infty$ and $m \to 0$, the set $\cR^{M,m}$ approaches the 
full realizable set $\cR$.
Related to the parameter $m$ is
\begin{align}
 p_0 := -\sup_{\substack{\bw \in \cR^{M, m} \\ \gamma \in (0, \gamma_0)
                         \\ v \in V}}
  \alphahatg(\bw) \cdot \bm(v) > 0;
\end{align}
for some $\gamma_0 \in (0, \infty)$.

The additional condition parameterized by $m$ in $\cR^{M,m}$ ensures that the 
ansatz $\Ga$ in the sublinear case is bounded away from zero.
In the superlinear case, the ans\"atze $\Ga$ for $\bsalpha \in \cR^M$ are 
already bounded away from zero, but this is not so for the sublinear case, 
where $\bsalpha$ must also fulfill $\bsalpha \cdot \bm(v) < 0$ for all
$v \in V$.

Note that $u_{M,m}$ and $h_{M,m}$ can be defined as in \eqref{eq:uM}:
\begin{align}
 u_{M,m} := \sup_{\bw \in \cR^{M,m}} \|\bw\|
  \qquand
 h_{M,m} := \sup_{\bw \in \cR^{M,m}} h(\bw);
\end{align}
both are finite.
We can also derive similar bounds on $\hg''$.
Let $\bc$ be the unit-length eigenvector associated with the largest eigenvalue 
of $(\hg)''_*(\alphahatgu)$ for some $\bu \in \cR^{M,m}$
\begin{subequations}
\begin{align}
 \lambda_{\max}((\hg)''_*(\alphahatgu)) &= \bc \cdot \left(
   \Vint{\bm \bm^T \etad''(\alphahatgu \cdot \bm)} + \gamma I \right) \bc \\
  &= \Vint{(\bc \cdot \bm)^2
   \etad''(\alphahatgu \cdot \bm)} + \gamma \\
  &\le |V| \left( \sup_{y \in [-M, -p_0]} \etad''(y) \right) + \gamma_0
\end{align}
\end{subequations}
for $\gamma \le \gamma_0$.
Similarly, if we now let $\bc$ be the unit-length eigenvector associated with 
the smallest eigenvalue of $(\hg)''_*(\alphahatgu)$ for $\bu \in \cR^{M,m}$ we 
have
\begin{align}
 \lambda_{\min}((\hg)''_*(\alphahatgu))
  = \Vint{(\bc \cdot \bm)^2 \etad''(\alphahatgu \cdot \bm)} + \gamma
  \ge |V| \inf_{y \in [-M, -p_0]} \etad''(y).
\end{align}

Thus for the sublinear case there exist positive constants
$\lambda_{\min, h'', M, m}$
and $C_{h'',  M, m}$ such that
\begin{align}\label{eq:hess-bnds-sub}
 \aligned
  \bv \cdot \hg''(\bu)\bv &\ge \lambda_{\min, h'', M, m} \|\bv\|^2 \\
  \|\hg''(\bu)\| &\le C_{h'', M, m}
 \endaligned
 \quad
 \text{for all } \bv \in \R^{N + 1} \text{, } \bu \in \cR^M
  \text{, and } \gamma \in (0, \gamma_0);
\end{align}
(cf.\ the superlinear case \eqref{eq:hess-bnds}).


\begin{lemma}\label{lem:bnds-sub}
Let $\eta$ be a sublinear kinetic entropy function, $M \in (0, \infty)$,
$m \in (0, \infty)$, and $\gamma_0 \in (0, \infty)$.
Then there exist positive constants $C_J$, $D_J$, $C_{\bF}$, and $D_{\bF}$ such 
that
\begin{equation}\label{eq:bnds-lem-sub}
 \aligned
 \|\fg(\ug | \bu)\| &\le C_{\bff} \hg(\ug | \bu)
   \\
 \|J_\gamma(\ug, \bu)\| &\le C_{J} \hg(\ug | \bu) + D_{J} \gamma^2
   \\
 q_\gamma(\ug, \bu) &\le C_q \hg(\ug | \bu) + D_q \gamma^2 
 \endaligned
 \quad \forall \ug \in \R^{N+1}, \bu \in \cR^{M, m}, \gamma \in (0,\gamma_0).
 \end{equation}
\end{lemma}

\begin{proof}
Because it is so similar to the superlinear case, we only sketch the 
proof for the sublinear case.

For sublinear $\eta$, the estimates of $\fg(\ug | \bu)$, $J_\gamma(\ug, \bu)$, 
$q_\gamma(\ug, \bu)$ and $\hg(\ug | \bu)$ on ${B_K \times \cR^{M,m}}$ can be 
derived just as in the superlinear case for ${B_K \times \cR^M}$, as well as 
the estimates of $\fg(\ug | \bu)$, $J_\gamma(\ug, \bu)$, and $q_\gamma(\ug, 
\bu)$ for large $\ug$ in 
\eqref{eq:fg-rel-iii-a}, \eqref{eq:J-iii-a}, \eqref{eq:q-iii-a}, 
\eqref{eq:fg-rel-iii-b}, \eqref{eq:J-iii-b}, and \eqref{eq:q-iii-b}.
The lower bound of $\hg(\ug | \bu)$, however, in \eqref{eq:hg-lb-iii-a} is no 
longer possible when $\eta$ is not bounded from below, and \lemref{fz} can no 
longer be applied to \eqref{eq:hg-lb-iii-b} because the right-hand side does 
not grow as $\u0hat(\alphahatg(\ug)) \to \infty$.

In the subcase $\|\uhat(\alphahatg(\ug))\| \le \delta \|\ug\|$, we first apply 
a combination of the bounds on individual terms from above to get
\begin{subequations}
\begin{align}
 \hg(\ug | \bu) &\ge |V| \eta\left(\frac1{|V|} \u0hat(\alphahatg(\ug))\right)
   + \frac{(1 - \delta)^2}{2\gamma_0} \|\ug\|^2 - h_{M,m} - M \|\ug\|
   - M u_{M,m}.
\end{align}
Now we recognize that $\u0hat(\alphahatg(\ug)) \le \|\uhat(\alphahatg(\ug))\|$
and then use that $\eta$ is a monotonically decreasing function to conclude
\begin{align}
 \hg(\ug | \bu) &\ge |V|\eta\left(\frac\delta{|V|} \|\ug\| \right)
   + \frac{(1 - \delta)^2}{2\gamma_0} \|\ug\|^2 - h_{M,m} - M \|\ug\|
   - M u_{M,m}.
\end{align}
\end{subequations}
Now thanks to the convexity of $\eta$, we can apply \lemref{fz} with
\begin{align}
 f(z) = |V|\eta\left(\frac{\delta z}{|V|}\right)
  + \frac{(1 - \delta)^2 z^2}{2 \gamma_0} - M z.
\end{align}

In the other subcase, where ${\|\uhat(\alphahatg(\ug))\| > \delta \|\ug\|}$, we 
use the assumption \eqref{eq:ass-Mm} and the first-order necessary condition 
\eqref{eq:1st-ord-necc} to get
\begin{subequations}
\begin{align}
 \hg(\ug | \bu) &= h(\uhat(\alphahatg(\ug)))
   + \frac\gamma 2 \|\alphahatg(\ug)\|^2
   - h(\uhat(\alphahatgu)) - \frac\gamma 2 \|\alphahatgu\|^2 \nonumber \\
  &\qquad - \alphahatgu \cdot \big( \uhat(\alphahatg(\ug))
    + \gamma \alphahatg(\ug)
    - \uhat(\alphahatgu)
    - \gamma \alphahatg(\bu)
    \big) \\
  &= h(\uhat(\alphahatg(\ug))) - h(\uhat(\alphahatgu))
   - \alphahatgu \cdot \left( \uhat(\alphahatg(\ug))
    - \uhat(\alphahatgu) \right) \nonumber \\
  &\qquad + \frac\gamma 2 \left\| \alphahatg(\ug) - \alphahatgu \right\|^2 \\
  &\ge |V|\eta\left(\frac1{|V|} \u0hat(\alphahatg(\ug))\right) - h_{M,m}
   - \alphahatgu \cdot \uhat(\alphahatg(\ug)) - M u_{M,m} \\
  &= |V|\eta\left(\frac1{|V|} \u0hat(\alphahatg(\ug))\right) - h_{M,m}
   - \Vint{\alphahatgu \cdot \bm \etad'(\alphahatg(\ug) \cdot \bm)} - M u_{M,m} 
   \\
  &\ge |V|\eta\left(\frac1{|V|} \u0hat(\alphahatg(\ug))\right) - h_{M,m}
   + p_0 \u0hat(\alphahatg(\ug)) - M u_{M,m}.
\end{align}
\end{subequations}
In the last step we have used $\etad' \ge 0$.
We are again ready to apply \lemref{fz} with
${f(z) = |V|\eta(z / |V|) + p_0 z}$ to derive conditions on $K$, $C_{\bff}$, and 
$C_J$ to achieve the desired estimate \eqref{eq:bnds-lem-sub}.

\end{proof}

This immediately gives the following corollary.

\begin{cor}\label{cor:sub}
Let $\eta$ be a sublinear kinetic entropy function and $\bu$ a Lipschitz continuous solution of 
the entropy-based moment equations \eqref{eq:moment-eq} for which there exist $M \in (0, \infty)$ and $m \in (0, \infty)$ so that
$\bu(t, x) \in \cR^{M,m}$ 
for all $(t, x) \in [0, T] \times X$.
Let $\{\ug\}_{\gamma \in (0,\gamma_0)}$ be a family of entropy solutions of \eqref{eq:reg-moment-eq} for
$\gamma \in (0, \gamma_0)$.
Furthermore assume that
$\|\nabla_x \bu\|_{L^\infty([0,T] \times X)} \le C_u$. 

Then
${\|\nabla_x \alphahatgu\|_{L^\infty([0,T] \times X)} \le C_{h'', M, m} 
C_u}$, and
\begin{align}
 \int_X \hg(\ug(T, x) | \bu(T, x)) \intdx \le \exp(C T) D T \gamma^2
\end{align}
for $C := C_{\alphahat} C_{\bff}+ C_u C_J + C_q$ and
$D := C_u D_J + D_q$, where the constants $C_{\bff}$, $C_J$, $C_q$, 
$D_J$, and $D_q$ are given by \lemref{bnds-sub}.
\end{cor}

Under the assumptions of Corollary~\ref{cor:sub} a result analogous to that of 
Corollary~\ref{cor:super-L2} is easy to prove.

\section{Numerical Results}
\label{sec:num-results}

We consider the toy problem from \cite{Hauck2010,AllFraHau19}.
There the authors considered the moment equations for the linear kinetic 
equations in slab geometry (see e.g., \cite{Lewis-Miller-1984}):
\begin{align}\label{eq:slab}
 \partial_t f + v \partial_x f = \sig{s}\left( \frac12 \Vint{f} - f \right).
\end{align}
where $V = [-1, 1]$ and $\sig{s} \in [0, \infty)$.
For the spatial domain we take $X = [0, 1]$.

\begin{remark}
More generally, the kinetic equation includes terms for absorption of particles 
by a background medium as well as a source term, as in 
\cite{Hauck2010,AllFraHau19}.
While the resulting terms in the moment equations can be straightforwardly 
incorporated into our analysis and do not affect our main results, we have left 
them out in this work for clarity of exposition.
\end{remark}

The corresponding entropy-based moment equations are
\begin{align}\label{eq:mn-slab}
 \partial_t \bu + \partial_x \bff(\bu) = \sig{s}R\bu,
\end{align}
where $R = \diag\{0, -1, \dots , -1\}$.
For the basis functions we used the Legendre polynomials.
We used the Maxwell--Boltzmann entropy \eqref{eq:mb}.
The collision term $\br(\bu) = \sig{s} R \bu$ clearly satisfies
\assref{r-lip}.

For numerical computations we used the fourth-order Runge--Kutta discontinuous 
Galerkin (RKDG) method as in \cite{AllFraHau19} with 160 spatial cells and no 
slope limiter.
With this spatial resolution the numerical solutions were accurate enough to 
observe the convergence in $\gamma$.

The initial conditions we used are constructed as follows.
Let ${\omega(x) := \frac12 M_0(1 + \cos(2 \pi x))}$ be a periodic function 
which we 
use to define the multiplier vector
\begin{align}
 \bsbeta(x) = \begin{pmatrix}
                  \log\left( \frac{\omega(x)}{2 \sinh(\omega(x))} \right) \\
                  \omega(x) \\
                  0 \\
                  \vdots \\
                  0
                 \end{pmatrix}
\end{align}
Then the initial conditions are given by
\begin{align}
 \bu^0(x) = \Vint{\bm \exp(\bsbeta(x) \cdot \bm)}.
\end{align}
Note that $\beta_0$ is chosen such that the zeroth order moment of the initial 
condition is one, i.e., $u^0_0(x) \equiv 1$.
The solution $\bu(t, x)$ of the original entropy-based moment equations with 
these initial conditions satisfies the assumption ${\bu(t, x) \in \cR^M}$ of 
Corollary~\ref{cor:super} for $M \approx M_0$.
Indeed the maximum value of $\|\alphahat(\bu(t, x))\|$ in space tends to 
decrease as time advances depending on the value of $\sig{s}$: the larger 
$\sig{s}$ is, the faster the norms of the multipliers decrease in time.
For $\sig{s} = 0$, the value of $\max_x \|\alphahat(\bu(t, x))\|$ is nearly 
constant in time.

We ran the solutions until the final time $T = 0.1$.
We used various values of $N$ up to 15 and found that the results did not 
depend qualitatively on the value of $N$.
We tried several values of $\sig{s}$ from zero to one and here did observe that 
the results depended on the value of $\sig{s}$:
for very small values of $\sig{s}$ the solutions appear not to enter the 
regime of second-order convergence until $\gamma$ is very small, on the order 
of $10^{-11}$.
Such values of $\gamma$ are so small, that for these solutions the error due to 
the numerical optimizer started to dominate errors due to the regularization.

We remind the reader that in order to evaluate the flux function $\fg$ we 
compute the multiplier vector $\alphahatg$ by numerically solving the dual 
problem \eqref{eq:reg-mult}.
We used the numerical optimizer described in \cite{AllFraHau19} but found that 
for problems with $\sig{s} = 0$ the value of tolerance $\tau$ on the norm of the 
dual gradient used in the stopping criterion, namely $\tau = 10^{-7}$, was not 
small enough to observe convergence in $\gamma$ for the very small values of 
$\gamma$ where the equations enter the regime of second-order convergence.
The difficulty here is that, in our experience, one cannot reliably 
bring the norm of the dual gradient below $10^{-7}$ when the norm of the 
multipliers at the solution is about ten or bigger.

But for smaller values of $M_0$, we found that using a combination of the 
smaller tolerance $\tau = 10^{-8}$ as well as modifying the optimizer to make 
efforts to further reduce the norm of the dual gradient when possible allowed 
us to decrease the numerical errors from the optimizer enough so that we could 
observe near second-order convergence.
This modification is described in pseudocode in Algorithm~\ref{alg:modstop} and 
works as follows:
The optimizer runs as usual until the norm of the dual gradient is smaller than 
$\tau$.
Then, the optimizer continues to take up to $\ell_{\max}$ additional iterations 
to bring the norm of the dual gradient under the smaller tolerance ${\tau_d \in 
(0, \tau)}$, which we call the \emph{desired} tolerance.
If the optimizer is unable to bring the norm of the dual gradient under 
$\tau_d$ in $\ell_{\max}$ additional iterations, the optimizer still exits 
successfully (as long as the current multiplier vector still satisfies the 
original stopping criterion).
In all of the results reported here, we used ${\tau_d = 10^{-11}}$ and 
${\ell_{\max} = 10}$.

\begin{algorithm}
\caption{The optimizer with modified stopping criterion}
\label{alg:modstop}
\begin{algorithmic}
\STATE $k \gets 0$
\STATE $\ell \gets 0$
\STATE ${\tt acceptable\_tolerance\_achieved} = $ false
\WHILE{$k < k_{\max}$}
 \IF{($\| \uhat(\bsalpha_k) + \gamma \bsalpha_k - \bu \| < \tau_d$)
     or ($\| \uhat(\bsalpha_k) + \gamma \bsalpha_k - \bu \| < \tau$
         and $\ell > \ell_{\max}$)}
  \RETURN $\bsalpha_k$
 \ENDIF
 \IF{${\tt acceptable\_tolerance\_achieved} = $ false
           and $\| \uhat(\bsalpha_k) + \gamma \bsalpha_k - \bu \| < \tau$}
  \STATE ${\tt acceptable\_tolerance\_achieved} \gets $ true
 \ENDIF
 \STATE Compute search direction $\bd_k$
 \STATE Perform backtracking line search to determine backtracking parameter 
  $\xi_k$
 \STATE $\bsalpha_{k + 1} \gets \bsalpha_k + \xi_k \bd_k$
 \STATE $k \gets k + 1$
 \IF{${\tt acceptable\_tolerance\_achieved} = $ true}
  \STATE $\ell \gets \ell + 1$
 \ENDIF
\ENDWHILE
\end{algorithmic}
\end{algorithm}

The results are given in Tables~\ref{tab:sig-1}~to~\ref{tab:sig-0}, which 
include the errors between the solution of the original equations and the 
solutions of the regularized equations measured in the relative entropy as well 
as in the $L^2$ and $L^\infty$ norms.
The error measured in the relative entropy is
\begin{align}
 \cH_\gamma(\ug | \bu) := \int_X \hg(\ug(T, x) | \bu(T, x))
  \intdx,
\end{align}
where we compute $\hg(\ug | \bu)$ numerically using the formula
\begin{align}\label{eq:rel-hg-num}
 \hg(\ug | \bu) = \Vint{\eta(G_{\alphahatg(\ug)}|G_{\alphahatg(\bu)})}
  + \frac \gamma 2 \| \alphahatg(\ug) - \alphahatg(\bu) \|^2,
\end{align}
where
\begin{subequations}
\begin{align}
 \eta(\Ga | \Gb) :=&\; \eta(\Ga) - \eta(\Gb) - \eta'(\Gb)(\Ga - \Gb) \\
  =&\; \eta(\Ga) - \eta(\Gb) - (\bsbeta \cdot \bm)(\Ga - \Gb).
\end{align}
\end{subequations}
(For the second line we have used $\Gb = \etad'(\bsbeta \cdot \bm)$ and
$\eta' \circ \etad' = \id$.)
Formula \eqref{eq:rel-hg-num} for $\hg(\ug | \bu)$ can be deduced by inserting 
\eqref{eq:hg-jg} into \eqref{eq:rel-hg} and simplifying, and it ensures the 
positivity of $\hg(\ug | \bu)$ despite errors due to the approximate 
computation of $\alphahatgu$.
The spatial integrals are computed using an eight-point Gauss quadrature on 
each of four subintervals in each spatial cell.
The observed convergence order $\nu$ between solutions computed with $\gamma_1$ 
and $\gamma_2$ is given by
\begin{align}
 \frac{\cH_{\gamma_1}(\bu_{\gamma_1} | \bu)}
      {\cH_{\gamma_2}(\bu_{\gamma_2} | \bu)}
  = \left( \frac{\gamma_1}{\gamma_2} \right)^\nu
\end{align}
The $L^2$ norm is computed using the same spatial quadrature, and the 
$L^\infty$ norm is approximated by taking the maximum over these spatial 
quadrature points.

In Tables~\ref{tab:sig-1}~and~\ref{tab:sig-001}, second-order convergence is 
clear in the relative entropy until the value of the relative entropy reaches 
about $10^{-17}$, which is below machine precision.
These tables include varying values of $M_0$ and $N$.
For $\sig{s} \ge 10^{-5}$, we observed second-order 
convergence for all values of $N$ up to 15 that we tried and for $M_0$ up to 
200.
For values of $M_0$ larger than 200, it is too difficult to satisfy the smaller 
optimization tolerance $\tau = 10^{-8}$.
In all cases, we observe first-order convergence in the $L^2$ norm as well as 
the $L^\infty$ norm.

For $\sig{s} = 0$, we were only able to solve the equations for smaller values 
of $M_0$ and observed second-order convergence for a smaller range of values of 
$\gamma$.
These results can be found in Table~\ref{tab:sig-0}, where we have included 
results from additional values of $\gamma$ between $10^{-9}$ and $10^{-11}$ to 
highlight the regime of second-order convergence.
In Table~\ref{tab:sig-0}, we see that the observed convergence orders increase 
monotonically to 1.99 and stay there until the value of the relative entropy 
goes below machine precision and the $L^\infty$ norm is smaller than the 
optimization tolerance $\tau = 10^{-8}$.
Indeed, at the final time, the optimizer was not able to solve many of the 
problems to the desired tolerance $\tau_d = 10^{-11}$, and in many of these 
problems, the tolerance $\tau$ is only barely fulfilled.
Therefore errors on the order of $10^{-8}$ in the $L^2$ and $L^\infty$ norms 
are not surprising, and this error of course also affects the computation of 
the relative entropy.

\begin{table}
\centering
\begin{tabular}{lcrcrcr}

$\gamma$ & $\cH_\gamma$ & $\nu$ & $L^2$ & $\nu$
         & $L^\infty$ & $\nu$ \\ \midrule

$10^{-3}$  & 3.977e-05 &   --  & 2.010e-03 &   --  & 3.939e-03 &   --  \\
$10^{-4}$  & 5.319e-07 &  1.87 & 2.115e-04 &  0.98 & 3.970e-04 &  1.00 \\
$10^{-5}$  & 5.439e-09 &  1.99 & 2.123e-05 &  1.00 & 3.948e-05 &  1.00 \\
$10^{-6}$  & 5.454e-11 &  2.00 & 2.131e-06 &  1.00 & 3.971e-06 &  1.00 \\
$10^{-7}$  & 5.504e-13 &  2.00 & 2.136e-07 &  1.00 & 3.667e-07 &  1.03 \\
$10^{-8}$  & 5.557e-15 &  2.00 & 2.147e-08 &  1.00 & 4.422e-08 &  0.92 \\
$10^{-9}$  & 5.886e-17 &  1.98 & 2.336e-09 &  0.96 & 5.671e-09 &  0.89 \\
$10^{-10}$ & 2.082e-18 &  1.45 & 3.110e-10 &  0.88 & 1.002e-09 &  0.75

\end{tabular}
\caption{Convergence test: $N = 9$, $\sig{s} = 1$, $M_0 = 100$.}
\label{tab:sig-1}
\end{table}

%
%
%

\begin{table}
\centering
\begin{tabular}{lcrcrcr}

$\gamma$ & $\cH_\gamma$ & $\nu$ & $L^2$ & $\nu$
         & $L^\infty$ & $\nu$ \\ \midrule

$10^{-3}$  & 1.196e-03 &   --  & 7.040e-03 &   --  & 2.040e-02 &   --  \\
$10^{-4}$  & 2.769e-04 &  0.64 & 2.490e-03 &  0.45 & 1.587e-02 &  0.11 \\
$10^{-5}$  & 3.442e-05 &  0.91 & 7.394e-04 &  0.53 & 6.236e-03 &  0.41 \\
$10^{-6}$  & 1.691e-06 &  1.31 & 1.336e-04 &  0.74 & 2.020e-03 &  0.49 \\
$10^{-7}$  & 2.033e-08 &  1.92 & 1.475e-05 &  0.96 & 2.168e-04 &  0.97 \\
$10^{-8}$  & 2.043e-10 &  2.00 & 1.481e-06 &  1.00 & 2.164e-05 &  1.00 \\
$10^{-9}$  & 2.130e-12 &  1.98 & 1.504e-07 &  0.99 & 2.199e-06 &  0.99 \\
$10^{-10}$ & 2.362e-14 &  1.96 & 1.569e-08 &  0.98 & 2.287e-07 &  0.98 \\
$10^{-11}$ & 2.380e-16 &  2.00 & 1.639e-09 &  0.98 & 2.415e-08 &  0.98

\end{tabular}
\caption{Convergence test: $N = 5$, $\sig{s} = 0.01$, $M_0 = 150$.}
\label{tab:sig-001}
\end{table}

%
%

\begin{table}
\centering
\begin{tabular}{lcrcrcr}

$\gamma$ & $\cH_\gamma$ & $\nu$ & $L^2$ & $\nu$
         & $L^\infty$ & $\nu$ \\ \midrule
$10^{-3}$     & 9.833e-06 &   --  & 1.112e-03 &   --  & 1.738e-03 &   --  \\
$10^{-4}$     & 8.398e-07 &  1.07 & 1.400e-04 &  0.90 & 2.146e-04 &  0.91 \\
$10^{-5}$     & 6.116e-08 &  1.14 & 1.634e-05 &  0.93 & 2.483e-05 &  0.94 \\
$10^{-6}$     & 3.136e-09 &  1.29 & 1.797e-06 &  0.96 & 2.636e-06 &  0.97 \\
$10^{-7}$     & 1.092e-10 &  1.46 & 1.885e-07 &  0.98 & 2.674e-07 &  0.99 \\
$10^{-8}$     & 2.255e-12 &  1.69 & 1.911e-08 &  0.99 & 2.679e-08 &  1.00 \\
$10^{-9}$     & 2.639e-14 &  1.93 & 1.916e-09 &  1.00 & 2.679e-09 &  1.00 \\
$10^{-9.25}$  & 8.420e-15 &  1.98 & 1.077e-09 &  1.00 & 1.509e-09 &  1.00 \\
$10^{-9.5}$   & 2.676e-15 &  1.99 & 6.060e-10 &  1.00 & 8.477e-10 &  1.00 \\
$10^{-9.75}$  & 8.499e-16 &  1.99 & 3.409e-10 &  1.00 & 4.776e-10 &  1.00 \\
$10^{-10}$    & 2.704e-16 &  1.99 & 1.919e-10 &  1.00 & 3.017e-10 &  0.80 \\
$10^{-10.25}$ & 8.711e-17 &  1.97 & 1.087e-10 &  0.99 & 1.886e-10 &  0.82 \\
$10^{-10.5}$  & 2.950e-17 &  1.88 & 6.189e-11 &  0.98 & 1.391e-10 &  0.53 \\
$10^{-10.75}$ & 9.893e-18 &  1.90 & 3.652e-11 &  0.92 & 1.182e-10 &  0.28 \\
$10^{-11}$    & 4.185e-18 &  1.49 & 2.256e-11 &  0.84 & 9.941e-11 &  0.30 \\
$10^{-12}$    & 1.801e-18 &  0.37 & 1.376e-11 &  0.21 & 6.393e-11 &  0.19

\end{tabular}
\caption{Convergence test: $N = 15$, $\sig{s} = 0$, $M_0 = 8$.}
\label{tab:sig-0}
\end{table}

\section{Conclusions}
\label{sec:conc}

The regularized entropy-based moment method for kinetic equations 
keeps many of the desirable properties of the original entropy-based moment 
method but removes the requirement that the moment vector of the solution 
remains realizable.
This facilitates the design and implementation of high-order numerical methods 
for the regularized moment equations.
However, the regularized equations require the selection of a regularization 
parameter, and the error caused by regularization needs to be accounted for and 
balanced with other error sources.
Our contribution is to rigorously prove the convergence as the regularization 
parameter goes to zero expected by formal arguments and to provide convergence 
rates.
Numerical experiments show that these rates are indeed optimal.

Our results hold for wide classes of entropy functions including the 
Maxwell--Boltzmann entropy and the Bose--Einstein entropy.
Our analysis relies on some key assumptions: The solution to the original 
moment equations needs to be Lipschitz and bounded away from the boundary of the 
set of realizable states.

Relaxing our assumptions would of course strengthen our results.
One would like to be able to work with kinetic equations with unbounded 
velocity domains, but here the original moment equations have fundamental 
problems \cite{Junk-1998,Jun00,Hauck-Levermore-Tits-2008} which remain in the 
regularized equations.
Nevertheless, the Euler equations, which are a case of the entropy-based moment 
method, do not have these problems and would be an interesting starting point 
for extending our analysis.
One would also like to allow the solution to have values arbitrarily close to or even on
the boundary of the realizable set, but not enough work has been done to 
consider the behavior of the moment equations near or on the boundary of the 
realizable set, such as in \cite{Coulombell-Goudon-2006}.
It is not even known whether the realizable set is invariant under the time 
evolution of the original entropy-based moment equations.
Finally, requiring a Lipschitz continuous solution to the limiting system is typical for relative entropy estimates, see \cite{Dafermos2016}, and in multiple space dimensions this is connected with non-uniqueness of entropy solutions for certain moment systems such as the Euler equations.

\appendix

\section{Constants}

Here we list some the constants which play the most significant roles 
throughout the paper.
Each constant is a strictly positive real number.

\begin{description}
 \item[$C_0$] Used to control the norm of a realizable moment vector using 
  its zeroth entry: ${\|\bu\| \le C_0 u_0}$ for all $\bu \in \cR$, where $u_0$ 
  is the zeroth component of $\bu$.  Introduced in \eqref{eq:C0}.
  
 \item[$C_1$] Global Lipschitz constant of $\bff$.  See 
  \eqref{eq:f-jacobian-bnd}.
  
 \item[$C_2$] Global Lipschitz constant of $\uhat \circ \alphahatg$; see 
  \eqref{eq:C2}.
  
 \item[$C_3$] $\sup_{\gamma \in (0, \gamma_0)} \|\uhat(\alphahatg(0)\|$, see
  \eqref{eq:C3}.
  This is used to get the affine bound in \eqref{eq:uhatag-bnd}.
  
 \item[$C_4$] Used when bounding $J_\gamma$ for
  $(\ug, \bu) \in \cR^L \times \cR^M$.
  Equal to $C_1 C_{h'',\max,M} + C_{\bff'', M} M C_2$; see \eqref{eq:lip-fg'}.
  
 \item[$u_M$] Upper bound on $\| \bu \|$ in $\cR^M$; see \eqref{eq:hM-uM}.
 
 \item[$h_M$] Upper bound on $h(\bu)$ in $\cR^M$; see \eqref{eq:hM-uM}.
 
 \item[$\lambda_{\min, h'', M}$] Lower bound on the smallest 
  eigenvalue of $\hg''$ over $\cR^M$ and $\gamma \in (0, \gamma_0)$; see 
  \eqref{eq:hess-bnds} and for the corresponding constant in the sublinear case 
  \eqref{eq:hess-bnds-sub}.
  
 \item[$C_{h'', M}$] Upper bound on $\|\hg''\|$ over $\cR^M$ and
  $\gamma \in (0, \gamma_0)$; see \eqref{eq:hess-bnds} and for the 
  corresponding constant in the sublinear case \eqref{eq:hess-bnds-sub}.
  
 \item[$C_{\bff'', L}$] Bound on the $\fg''$ over $\conv(\cR^L)$ and
  $\gamma \in (0, \gamma_0)$; see \eqref{eq:C-fg''}.
  
  
  
 \item[$C_{h, M, L}$] Lower bound on $\hg(\bv | \bu)$ for 
  $(\bv, \bu) \in (\R^{N + 1} \setminus \cR^L) \times \cR^M$
  see \lemref{C-M-L-positive} and Appendix~\ref{sec:C-M-L-positive}.
  
  
  
 \item[$C_{\br}$] Lipschitz constant for $\br$; see \assref{r-lip}.
\end{description}

\section{Entropy relationships}
\label{sec:mn-review}

In this appendix we quickly review the computations from \cite{Lev96} showing 
the relationships between $h$, $h_*$, $\uhat$, and $\alphahat$, in particular 
the key result that $h' = \alphahat$.
Start with the definition of the entropy $h$,
\begin{align}
 h(\bu) := \min_{g \in \bbF(V)} \left\{ \Vint{\eta(g)}
  : \Vint{\bm g} = \bu \right\},
\end{align}
i.e., the minimal value of the primal problem \eqref{eq:primal} as a function 
of the moment vector.
The corresponding Lagrangian is given by
\begin{align}
 L(g, \bsalpha) := \Vint{\eta(g)}
  + \bsalpha \cdot \left( \bu - \Vint{\bm g}
  \right),
\end{align}
and thus the dual problem is
\begin{align}\label{eq:dual-appendix}
 \max_{\bsalpha \in \R^{N + 1}} \min_{g \in \bbF(V)} L(g, \bsalpha) = 
 \max_{\bsalpha \in \R^{N + 1}} \bsalpha \cdot \bu
  - \Vint{\etad(\bsalpha \cdot \bm)},
\end{align}
cf.\ \eqref{eq:dual}, where to get this equality one takes the minimization 
inside the integral and applies the definition of the Legendre dual of $\eta$.
Because the duality gap is zero \cite[Thm.\ 16]{Hauck-Levermore-Tits-2008} we 
have
\begin{align}
 h(\bu) = \max_{\bsalpha \in \R^{N + 1}} \left\{ \bsalpha \cdot \bu
  - \Vint{\etad(\bsalpha \cdot \bm)} \right\},
\end{align}
so $h$ is the Legendre transformation of
\begin{align}
 h_*(\bsalpha) := \Vint{\etad(\bsalpha \cdot \bm)}.
\end{align}
Its derivative is readily computed:
\begin{align}
 h'_*(\bsalpha) = \Vint{\bm \etad'(\bsalpha \cdot \bm)} =: \uhat(\bsalpha).
\end{align}

Now, recall that $\alphahatu$ in \eqref{eq:dual} is defined to be the 
multiplier vector that solves the dual problem \eqref{eq:dual-appendix}.
Then the first-order necessary conditions for the dual problem imply
\begin{align}
 \uhat(\alphahatu) = \bu.
\end{align}
The reverse, i.e., $\alphahat(\uhat(\bsalpha)) = \bsalpha$, is a consequence of 
the uniqueness of the solution to the dual problem (thanks to convexity).
Thus $\alphahat$ is the inverse function of $\uhat$.
Finally, since the derivative of Legendre duals are inverses of each other, we 
have
\begin{align}
 h' = (h'_*)^{-1} = \uhat^{-1} = \alphahat.
\end{align}

\section{The regularized solution for zero vector as $\gamma \to 0$}
\label{sec:uag0}

In this section we quickly consider
\begin{align}
 \lim_{\gamma \to 0} \uhat(\alphahatg(0)),
\end{align}
where $0 \in \R^{N + 1}$.
This comes up in \secref{basic} when deriving global estimates on the function
$\uhat \circ \alphahatg$ using Lipschitz continuity.

For convenience we assume that the basis functions are orthogonal to each 
other, which since $m_0 \equiv 1$ (recall \assref{basis-functions}) in 
particular implies that $\vint{m_i} = \vint{m_0 m_i} = 0$ for all
$i \in \{1, \ldots , N\}$.
As a consequence, most of the components of $\alphahatg(0)$ are easy to 
determine.
Consider the first-order necessary conditions:
\begin{align}
 0 = \Vint{m_i \etad'(\alphahatg(0) \cdot \bm)}
 + \gamma \hat{\alpha}_{\gamma, i}(0), \qquad i \in \{0, 1, \ldots , N\}.
\end{align}
If we set $\hat{\alpha}_{\gamma, i}(0) = 0$ for $i \in \{1, \ldots , N\}$, then 
the entropy ansatz is constant in $v$, and by orthogonality of $\{m_i\}$, we 
see that the first-order necessary conditions are satisfied for
$i \in \{1, \ldots , N\}$.
It remains to determine the zeroth component $\hat{\alpha}_{\gamma, 0}(0)$, for 
which we need to solve
\begin{align}
 0 = |V| \etad'(\hat{\alpha}_{\gamma, 0}(0))
  + \gamma \hat{\alpha}_{\gamma, 0}(0).
\end{align}
From this equation, it is clear that
$\hat{\alpha}_{\gamma, 0}(0) < 0$ and thus that
$\hat{\alpha}_{\gamma, 0}(0) \to -\infty$ monotonically (recall 
\eqref{eq:ag-dec-gamma}) as $\gamma \to 0$.
The limiting value must be unbounded because $0 \nin \cR$.
Now we recall that for the $\eta$ considered in this work we have
$\range(\etad') = \dom(\eta') = (0, \infty)$, and furthermore that $\etad'$ 
is a monotonically increasing function because $\etad$ is convex.
Therefore $\etad'(\hat{\alpha}_{\gamma, 0}(0)) \to 0$ as $\gamma \to 0$.
%
%
It follows that also ${\uhat(\alphahatg(0)) \to 0}$.

\section{Proof of \lemref{C-M-L-positive}}
\label{sec:C-M-L-positive}

Let $L \in (M, \infty)$.
We want to show
\begin{align}\label{eq:C-M-L-appendix}
 C_{h, M, L} := \inf_{\substack{\bv \in \R^{N + 1} \setminus \cR^L \\
                                \bu \in \cR^M \\
                                \gamma \in (0, \gamma_0)}}
                 \hg(\bv | \bu)
             > 0.
\end{align}
The basic idea is that, by strict convexity of $\hg(\bv | \bu)$ in its first 
argument, it only achieves its minimum value, zero, when $\bv = \bu$.
But this is ruled out on
${(\bv, \bu) \in (\R^{N + 1} \setminus \cR^L) \times \cR^M}$ because
${\overline{\R^{N + 1} \setminus \cR^L} \cap \cR^M = \emptyset}$. 

We can get a more explicit bound as follows.
Let $Q \in (M, L)$.
We claim that for every $\bv \in \R^{N + 1} \setminus \cR^L$ and
$\bu \in \cR^M$ there exists a $\lambda_Q \in (0, 1)$ such that
\begin{align}\label{eq:wQ}
 \bw_Q := (1 - \lambda_Q) \bu + \lambda_Q \bv \in \cR^L \setminus \cR^Q.
\end{align}
For $\bv \in \cR \setminus \cR^L$ this straightforward, because the function
\begin{align}
 f(\lambda) = \|\alphahat((1 - \lambda) \bu + \lambda \bv)\|
\end{align}
is a continuous function with $f(0) \le M$ and $f(1) \ge L$.
For $\bv \in \R^{N + 1} \setminus \cR$, by convexity of $\cR$ there exists a 
unique $\lambda_{\cR} \in (0, 1)$ such that
$(1 - \lambda_{\cR}) \bu + \lambda_{\cR} \bv \in \partial \cR$.
But then, since $\cR^L \subset \subset \cR$, there must also be a
$\lambda_L \in (0, \lambda_{\cR})$ such that 
$(1 - \lambda_L) \bu + \lambda_L \bv \in \cR \setminus \cR^L$, and so the first 
argument can be applied again.

Now, for any $(\bw, \bu) \in (\cR^L \setminus \cR^Q) \times \cR^M$ we have
\begin{align}
 \hg(\bw | \bu)
  \stackrel{\eqref{eq:rel-ent-bnd-below-L}}{\ge}
   \lambda_{\min, h'', \widetilde L} \|\bw - \bu\|^2
  \ge \lambda_{\min, h'', \widetilde L}
   \inf_{\substack{\bw \in \overline{\cR^L \setminus \cR^Q} \\
                   \bu \in \cR^M}}
    \|\bw - \bu\|^2
  =: C_L > 0
\end{align}
where the strict positivity follows from
$\cR^M \cap \overline{\cR^L \setminus \cR^Q} = \emptyset$.

Finally, let $(\bv, \bu) \in (\R^{N + 1} \setminus \cR^L) \times \cR^M$ and
$\bw_Q$ be as in \eqref{eq:wQ}.
By convexity of the relative entropy in its first argument, we have
\begin{align}
 (1 - \lambda_Q) \hg(\bu | \bu) + \lambda_Q \hg(\bv | \bu)
  \ge \hg(\bw_Q | \bu) \ge C_L.
\end{align}
But with $\hg(\bu | \bu) = 0$ and $\lambda_Q \in (0, 1)$, we immediately have
$\hg(\bv | \bu) \ge C_L$, and by taking the infimum as in 
\eqref{eq:C-M-L-appendix} we conclude $C_{h, M, L} \ge C_L$.

\bibliographystyle{plain}
\bibliography{refs-rmn-conv}

\end{document}